\documentclass{amsart}
\usepackage{amsmath,amsthm}
\usepackage{amsfonts,amssymb}
\usepackage{accents}
\usepackage{enumerate}
\usepackage{accents,color}
\usepackage{graphicx}

\usepackage{todonotes}

\newcommand{\lvt}{\left|\kern-1.35pt\left|\kern-1.3pt\left|}
\newcommand{\rvt}{\right|\kern-1.3pt\right|\kern-1.35pt\right|}

\newtheorem{thm}{Theorem}[section]
\newtheorem{cor}[thm]{Corollary}
\newtheorem{lem}[thm]{Lemma}
\newtheorem{prop}[thm]{Proposition}

\theoremstyle{remark}

\newcommand{\secref}[1]{\S\ref{#1}}

 \def\la{{\langle}}
 \def\ra{{\rangle}}

 \def\a{{\alpha}}
 \def\b{{\beta}}
 \def\g{{\gamma}}
 
 \def\t{{\theta}}
 
 \def\d{{\delta}}
 
 \def\s{\sigma}
 \def\la{{\langle}}
 \def\ra{{\rangle}}

 \def\CB{{\mathcal B}}

 \def\CH{{\mathcal H}}

 \def\CP{{\mathcal P}}

 \def\CV{{\mathcal V}}

 \def\RR{{\mathbb R}}

\def\lla{\langle{\kern-2.5pt}\langle}      
\def\rra{\rangle{\kern-2.5pt}\rangle}

\newcommand{\wh}{\widehat}

\def\f{\frac}

\def\cal#1{\mathcal{#1}}
\def\ds{{\rm d}s}

\def\I{{\rm i}}
\def\vc#1{{\mathbf #1}}
\graphicspath{{./}}
\begin{document}

\title{Orthogonal structure on a wedge and on the boundary of a square}

\author{Sheehan Olver}
\address{Department of Mathematics\\
Imperial College\\
 London \\
 United Kingdom  }\email{s.olver@imperial.ac.uk}

\author{Yuan Xu}
\address{Department of Mathematics\\ University of Oregon\\
    Eugene, Oregon 97403-1222.}\email{yuan@uoregon.edu}

\thanks{The second author was supported in part by NSF Grant DMS-1510296.}
\thanks{Communicated by Alan Edelman.}

\date{\today}

\begin{abstract} 
Orthogonal polynomials with respect to a weight function defined on a wedge in the plane are studied. A
basis of orthogonal polynomials is explicitly constructed for two large class of weight functions and the convergence
of Fourier orthogonal expansions is studied. These are used to establish analogous results for orthogonal polynomials 
on the boundary of the square. As an application, we study the statistics of the associated determinantal point process and use the basis to calculate Stieltjes transforms.
\end{abstract}

\maketitle

\section{Introduction}
\setcounter{equation}{0}

Let $\Omega$ be a wedge on the plane that consists of two line segments sharing a common endpoint. For 
a positive measure $d\mu$ defined on $\Omega$, we study orthogonal polynomials of two variables with respect to 
the bilinear form 
$$
       \la f, g\ra = \int_{\Omega} f(x,y) g(x,y) d\mu.
$$
We also study orthogonal polynomials on the boundary of a parallelogram. Without loss of generality we can 
assume that our wedge is of the form
\begin{equation}\label{eq:1.1}
    \Omega = \{(x_1,1): x_1 \in [0,1]\} \cup  \{(1,x_2): x_2 \in [0,1]\} 
\end{equation}
and consider the bilinear form defined by 
\begin{equation}\label{eq:1.2}
  \la f, g\ra = \int_0^1 f(x,1)g(x,1) w_1(x) dx + \int_0^1 f(1,y)g(1,y) w_2(y) dy.
\end{equation}

Since $\Omega$ is a subset of the zero set of a quadratic polynomial $l_1(x,y) l_2(x,y)$, where $l_1$ and $l_2$ 
are linear polynomials, the structure of orthogonal polynomials on $\Omega$ is very different from that of ordinary 
orthogonal polynomials in two variables \cite{DX} but closer to that of spherical harmonics. The latter are defined 
as homogeneous polynomials that satisfy the Laplace equation $\Delta Y = 0$ and are orthogonal on the unit circle, 
which is  the zero set of the  quadratic polynomial $x^2 + y^2-1$. The space of spherical polynomials of degree $n$ has 
dimension 2 for each $n \ge 1$ and, furthermore, one basis of spherical harmonics when restricted on the unit circle 
are $\cos  n \theta$ and $\sin n \theta$, in polar coordinates $(r,\t)$, and the Fourier orthogonal expansions in 
spherical harmonics coincide with the classical Fourier series. 

In \secref{Section:Wedge}, we consider orthogonal polynomials on a wedge. The space of orthogonal polynomials of degree $n$ has dimension 2 for each $n \ge 1$, just like that of 
spherical harmonics, and they satisfy the equation $\partial_1 \partial_2 Y = 0$.   The main results are
\begin{itemize}
\item An explicit expression in terms of univariate orthogonal polynomials when  $w_1(x) = w_2(x) =w(x)$ where $w$ is any weight function on $[0,1]$ (Theorem~\ref{thm:ipd-w-op}), 
\item  Sufficient conditions for  pointwise and uniform convergence    (Theorem~\ref{thm:pointwise}),  as well as normwise convergence (Corollary~\ref{cor:normwiseconv}),
\item Explicit expression in terms of Jacobi polynomials when $w_1(x) = w_{\a,\g}(x)$ and $w_2(x) = w_{\b,\g}(x)$ (Theorem~\ref{thm:OP-ipd-abg}),
\item  Sufficient conditions for  normwise convergence     (Theorem~\ref{thm:3.4}).
\end{itemize}


In \secref{Section:BoundarySquare} we study orthogonal polynomials on the boundary of a parallelogram, which we can assume as 
the square $[-1,1]^2$ without loss of generality. For a family of generalized Jacobi weight functions that are symmetric in 
both $x$ and $y$, we are able to deduce an orthogonal basis in terms of four families of orthogonal bases on the wedge
in Theorem~\ref{thm:boundaryOP}. Furthermore, the convergence of the Fourier orthogonal expansions can also be deduced in this fashion,
as shown in Theorem~\ref{thm:squareconv}. 

In \secref{Section:Square} we use orthogonal polynomials on the boundary of the square to construct an orthogonal 
basis for the weight function $w(\max\{|x|,|y|\})$ on the square $[-1,1]^2$. 
This mirrors the way in which spherical harmonics can be used to construct a basis of orthogonal polynomials for the weight function $w(\sqrt{x^2+y^2})$
on the unit disk.  However, unlike the unit disk, the orthogonal 
basis we constructed are no longer polynomials in $x,y$ but are polynomials of $x, y$ and $s=\max\{|x|,|y|\}$.


The study is motivated by applications. In particular, we wish to investigate how the applications of univariate orthogonal polynomials can be translated to multivariate orthogonal polynomials on curves. As a motivating example, univariate orthogonal polynomials give rise to a determinantal point process that is linked to the eigenvalues of unitary ensembles of random matrix theory. In \secref{sec:detpointprocess}, we investigate the statistics of the determinantal point process generated from  orthogonal polynomials on the wedge, and find experimentally that they have the same local behavior
as a Coulomb gas away from the corners/edges. 

In Appendix~\ref{sec:JacobiOperators}, we give the Jacobi operators associated with a special case of weights on the wedge, whose entries are rational. 
Finally, in Appendix~\ref{sec:Stieltjes} we show that the Stieltjes transform of our family of orthogonal polynomials satisfies a recurrence that can be built out of the Jacobi operators of the orthogonal polynomials, which can in turn be used to compute Stieltjes transforms numerically. This is a preliminary step towards using these polynomials for solving singular integral equations.

\section{Orthogonal polynomials on a wedge}\label{Section:Wedge}
\setcounter{equation}{0}

Let $\CP_n^2$ denote the space of homogeneous polynomials of degree $n$ in two variables; that is,  
$\CP_n^2 =\mathrm{span}\,  \{x^{n-k}y^k: 0 \le k \le n\}$. Let $\Pi_n^2$ denote the space of polynomials
of degree at most $n$ in two variables. 

\subsection{Orthogonal structure on a wedge}
Given three non-collinear points, we can define a wedge by fixing 
one point and joining it to other points by line segments. We are interested in orthogonal polynomials
on the wedge. Since the three points are non-collinear, each wedge can be mapped to 
$$ 
    \Omega = \{(x_1,1): x_1 \in [0,1]\} \cup  \{(1,x_2): x_2 \in [0,1]\}
$$
by an affine transform. Since the polynomial structure and the orthogonality are preserved under the affine 
transform, we can work with the wedge $\Omega$ without loss of generality. Henceforth we work only on
$\Omega$. 

Let $w_1$ and $w_2$ be two nonnegative weight functions defined on $[0,1]$. We consider the bilinear form 
define on $\Omega$ by
\begin{equation}\label{eq:ipd-w}
     \la f,g\ra_{w_1,w_2} := \int_0^1 f(x,1) g(x,1)w_1(x) dx +  \int_0^1 f(1,y) g(1,y)w_2(y) dy. 
\end{equation}
Let $I$ be the polynomial ideal of $\RR[x,y]$ generated by $(1-x)(1-y)$. If $f \in I$, then $\la f, g\ra_{w_1,w_2} =0$
for all $g$. The bilinear form defines an inner product on $\Pi_n^2$, modulo $I$, or equivalently, on the quotient
space $\RR[x,y]/I$. 

\begin{prop}
Let $\CH_n^2(w_1,w_2)$ be the space of orthogonal polynomials of degree $n$ in $\RR[x,y]/I$.  
Then 
$$
   \dim \CH_0^2(w_1,w_2)=1 \quad\hbox{and} \quad  \dim \CH_n^2(w_1,w_2) =2, \quad n \ge 1.
$$
Furthermore, we can choose polynomials in $\CH_n^2(w_1,w_2)$ to satisfy $\partial_x \partial_y p = 0$. 
\end{prop}

\begin{proof}
Since $(1-x)(1-y) \CP_{n-2}$ is a subset of $I$, the dimension of $\dim \CH_n^2(w_1,w_2) \le 2$. Applying 
the Gram--Schmidt process on $\{1, x^k, y^k, k \ge 1\}$ shows that there are two orthogonal 
polynomials of degree exactly $n$. Both these polynomials can be written in the form of 
$p(x) + q(y)$, since we can use $xy \equiv x+y-1$ mod $I$ to remove all mixed terms. Evidently 
such polynomials satisfy $\partial_x \partial_y (p(x) + q(y))=0$. 
\end{proof}

In the next two subsections, we shall construct an orthogonal basis of $\CH_n^2(w_1,w_2)$ for certain 
$w_1$ and $w_2$ and study the convergence of its Fourier orthogonal expansions. We will make use
of results on orthogonal polynomials of one variable, which we briefly record here. 

For $w$ defined on $[0,1]$, we let $p_n(w)$ denote an orthogonal polynomial of degree $n$ with respect to $w$,
and let $h_n(w)$ denote the norm square of $p_n(w)$, 
$$
  h_n(w) := \int_0^1 |p_n(w;x)|^2 w(x) dx.
$$ 
Let $L^2(w)$ denote the $L^2$ space with respect to $w$ on $[0,1]$. The Fourier orthogonal 
expansion of $f \in L^2(w)$ is defined by 
$$
  f = \sum_{n=1}^\infty \wh f_n(w) p_n(w) \quad \hbox{with} \quad  \wh f_n(w) = \frac1{h_n(w)} \int_0^1 f(y) p_n(w;y) w(y)dy. 
$$
The Parseval identity implies that 
$$
    \|f\|_{L^2(w,[0,1])}^2 = \sum_{n=0}^\infty \left|\wh f_n(w) \right|^2 h_n(w).
$$
The $n$-th partial sum of the Fourier orthogonal expansion with respect to $w$ can be written as an integral
\begin{equation} \label{eq:partial-sum}
  s_n(w;f)(x) := \sum_{k=1}^n \wh f_k(w) p_k(w;x) = \int_{-1}^1 f(y) k_n(w;x,y)w(y)dy, 
\end{equation}
where $k_n(w)$ denotes the reproducing kernel defined by 
\begin{equation} \label{eq:reprod-kernel}
   k_n(w;x,y) = \sum_{k=0}^n  \frac{p_k(w;x) p_k(w;y)}{h_k(w)}.
\end{equation}

\subsection{Orthogonal structure for $w_1=w_2$ on a wedge}
In the case of $w_1 = w_2 = w$, we denote the inner product \eqref{eq:ipd-w} by $\la \cdot,\cdot\ra_w$ and the space of
orthogonal polynomials by $\CH_n^2(w)$. In this case, an orthogonal basis for $\CH_n^2(w)$ can be constructed explicitly.

\begin{thm} \label{thm:ipd-w-op}
Let $w$ be a weight function on $[0,1]$ and let $\phi w(x): = (1-x)^2 w(x)$. Define 
\begin{align} \label{eq:PnQ}
\begin{split}
  P_n(x,y) & = p_n(w;x)+ p_n(w;y) - p_n(w;1), \quad n= 0,1,2,\ldots,\\
  Q_n(x,y) & =  (1-x) p_{n-1}(\phi w; x) - (1-y) p_{n-1}(\phi w; y), \quad n=1,2, \ldots. 
\end{split}
\end{align}
Then $\{P_n, Q_n\}$ are two polynomials in $\CH_n^2(w)$ and they are mutually orthogonal.  
Furthermore,
\begin{equation}\label{eq:PnQnorm}
  \la P_n, P_n\ra_w = 2 h_n(w) \quad\hbox{and}\quad  \la Q_n, Q_n\ra_w = 2 h_{n-1}(\phi w). \quad 
\end{equation}
\end{thm}

\begin{proof}
Since $P_n(x,1) = P_n(1,x)$ and $Q_n(x,1) = - Q_n(1,x)$, it follows that 
$$
\la P_n, Q_m \ra_{w} = \int_0^1 P_n(x,1) Q_m(x,1)  w(x)dx + \int_0^1 P_n(1,x) Q_m(1,x) w(x)dx =0
$$
for $n \ge 0$ and $m \ge 1$. Furthermore, 
$$
  \la P_n, P_m\ra_{w} = 2 \int_0^1 p_n(w;x) p_m(w; x) w(x) dx = 2 h_n(w)\delta_{n,m}
$$
by the orthogonality of $p_n(w)$. Similarly, 
$$
 \la Q_n, Q_m\ra_{w} = 2 \int_0^1 p_{n-1}(\phi w; x)p_{m-1}(\phi w; x)(1-x)^2 w(x)  dx = 2 h_{n-1}(\phi w) \delta_{n,m}. 
 $$
The proof is completed.
\end{proof}

Let $L^2(\Omega,w)$ be the space of Lebesgue measurable functions with finite 
$$
  \|f\|_{L^2(\Omega,w)}: = \la f,f\ra_w^{\f12} =  \left( \|f(\cdot, 1)\|_{L^2(w,[0,1])}^2 +  \|f(1,\cdot)\|_{L^2(w,[0,1])}^2\right)^{\f12}  
$$
norms. For $f \in L^2(\Omega, w)$, its Fourier expansion is given by 
$$
 f  =  \wh f_0 + \sum_{n=1}^\infty \left[  \wh f_{P_n}P_n + \wh f_{Q_n}Q_n  \right],
$$
where $P_n$ and $Q_n$ are defined in Theorem  \ref{thm:ipd-w-op} and
$$
 \wh f_0:= \frac{ \la f,1\ra_w}{\la 1,1\ra_w}, \qquad
 \wh f_{P_n}:= \frac{\la f,P_n \ra_w}{\la P_n,P_n \ra_w}, \qquad \wh f_{Q_n}:= \frac{\la f,Q_n \ra_w}{\la Q_n,Q_n \ra_w}.
$$
The partial sum operator $S_n f$ is defined by 
$$
  S_n f :=  \wh f_0 + \sum_{k=1}^n \left[  \wh f_{P_k}P_k + \wh f_{Q_k}Q_k  \right],
$$
which can be written in terms of an integral in terms of the reproducing kernel $K_n(\cdot,\cdot)$, 
$$
  S_n f(x_1,x_2) = \la f, K_n((x_1,x_2), \cdot) \ra_w,
$$
where 
$$
  K_n((x_1,x_2), (y_1,y_2)):= \frac{1}{\la 1,1\ra_w} + \sum_{k=1}^n \left[  \frac{P_k(x_1,x_2)P_k(y_1,y_2)}{\la P_k,P_k \ra_w}
       +  \frac{Q_k(x_1,x_2)Q_k(y_1,y_2)}{\la Q_k,Q_k \ra_w}  \right].
$$
We show that this kernel can be expressed, when restricted on $\Omega$, in terms of the reproducing kernel 
$k_n(w;\cdot,\cdot)$ defined at \eqref{eq:reprod-kernel}.

\begin{prop}\label{prop:repkernelsym}
The reproducing kernel $K_n(\cdot,\cdot)$ for $\la \cdot,\cdot\ra_w$ satisfies
\begin{align}
 K_n((x,1),(y,1))  &\, =  K_n((1,x),(1,y)) \\
                       &\, = \frac12 k_n(w;x,y) + \frac12 (1-x)(1-y)k_{n-1}(\phi w; x,y), \notag \\
 K_n((x,1),(1,y)) & \, = K_n((1,x),(y,1)) \\
      & \, = \frac12 k_n(w;x,y) - \frac12 (1-x)(1-y)k_{n-1}(\phi w; x,y).  \notag
\end{align}
\end{prop}

\begin{proof}
By \eqref{eq:PnQ} and \eqref{eq:PnQnorm},
\begin{align*}
 K_n((x,1),(y,1)) & \, = \frac{1} {2h_0(w)} +   \sum_{k=1}^n   \frac{p_k(w; x) p_k(w;y)}{2 h_k(w)} \\
             & \qquad\qquad 
                +  \sum_{k=1}^n \frac{(1-x)p_{k-1}(\phi w;x) (1-y) p_{k-1}(\phi w;y)}{ 2 h_{k-1}(\phi w)} \\
         &\, = \frac12 k_n(w;x,y) + \frac12 (1-x)(1-y)k_{n-1}(\phi w; x,y).
\end{align*}
The other case is established similar, using $Q_k(1,y) = - (1-y) p_{k-1}(\phi w; y)$.  
\end{proof}

It is well-known that the kernel $k_n(w; \cdot,\cdot)$ satisfies the Christoffel--Darboux formula, which plays an 
important role for the study of Fourier orthogonal expansion. Our formula allows us to write down an analogue of 
Christoffel--Darboux formula for $K_n(\cdot,\cdot)$, but we can derive convergence directly. 

\begin{thm} \label{thm:pointwise}
Let $f$ be a function defined on $\Omega$. Define 
$$
 f_e(x) := \frac12(f(x,1) + f(1,x)) \quad \hbox{and} \quad f_o(x) := \frac12\frac{f(x,1)-f(1,x)}{1-x}. 
$$
Then
\begin{align}
  S_n f(x_1,1) & \,= s_n(w; f_e, x_1) + (1-x_1) s_{n-1} (\phi w; f_o, x_1),  \label{eq:Sn=1}  \\
  S_n f(1,x_2) & \,= s_n(w; f_e, x_2) - (1-x_2) s_{n-1} (\phi w; f_o, x_2).  \label{eq:Sn=2} 
\end{align}
In particular, if $s_n(w;f_e,x) \to f_e(x)$ and $s_n(\phi w;f_o,x) \to f_o(x)$, pointwise or in the uniform norm 
as $n\to \infty$, then $S_n f(x)$ converges to $f(x)$ likewise. 
\end{thm}

\begin{proof}
By our definition,  
\begin{align*}
S_n f(x_1,1)  = &\, \int_0^1 f(y,1) K_n((x_1,1),(y,1)) w(y) dy +  \int_0^1 f(1,y) K_n((x_1,1),(1,y)) w(y) dy \\ 
        = &\, \frac12 \int_0^1 f(y,1) \left[ k_n(w; x_1,y)+ (1-x_1)(1-y) k_{n-1}(\phi w; x_1,y) \right] w(y) dy \\
         & \,+ \frac12 \int_0^1 f(1,y) \left[ k_n(w; x_1,y)-  (1-x_1)(1-y) k_{n-1}(\phi w; x_1,y) \right] w(y) dy \\
       = &\, s_n(w; f_e, x_1) + (1-x_1) s_{n-1} (\phi w;f_o,x_1).        
\end{align*}
Similarly, 
\begin{align*}
S_n f(1,x_2)  = &\, \int_0^1 f(y,1) K_n((1,x_2),(y,1)) w(y) dy +  \int_0^1 f(1,y) K_n((1,x_2),(1,y)) w(y) dy \\ 
        = &\, \frac12 \int_0^1 f(y,1) \left[ k_n(w; x_2,y) - (1-x_2)(1-y) k_{n-1}(\phi w; x_2,y) \right] w(y) dy \\
         & \,+ \frac12 \int_0^1 f(1,y) \left[ k_n(w; x_2,y) + (1-x_2)(1-y) k_{n-1}(\phi w; x_2,y) \right] w(y) dy \\
       = &\, s_n(w; f_e, x_2) - (1-x_2) s_{n-1} (\phi w;f_o,x_2).        
\end{align*}
Moreover, since $f_e(x) + (1-x) f_o(x) = f(x,1)$ and $f_e(x) - (1-x) f_o(x) = f(1,x)$, it follows that 
\begin{align*}
  S_n f (x_1,1) - f(x_1) = &\, s_n(w; f_e, x_1) - f_e(x_1)+ (1-x_1) \left(s_{n-1} (\phi w;f_o, x_1) - f_o(x_1) \right),\\
  S_n f (1,x_2) - f(x_2) = &\, s_n(w; f_e, x_2) - f_e(x_2)- (1-x_2) \left(s_{n-1} (\phi w;f_o, x_2) - f_o(x_2) \right)
\end{align*}
from which we see that the convergence of $s_n(w;f_e)$ and $s_n(\phi w; f_o)$ imply the convergence of
$S_n f$. 
\end{proof}

Since $f\in L^2(\Omega,w)$, it is evident that $f_e \in L^2(w)$. Moreover, $f_o \in L^2(\phi w)$ since
$$
  \int_0^1 |f_o(x)|^2 \phi w(x) dx = \int_0^1 |f(x,1) - f(1,x)|^2 w(x) dx \le 2 \|f\|_{L^2(\Omega, w)}^2. 
$$
In particular, $s_n (w, f_e)$ and $s_n(\phi w; f_o)$ converge to $f_e$ and $f_o$ in $L^2(w)$ and in $L^2(\phi w)$, 
respectively. 

\begin{cor}\label{cor:normwiseconv}
If $f\in L^2(\Omega, w)$, then 
$$
\|f - S_n(f) \|_{L^2(\Omega,w)}^2 = 2 \left( \|s_{n} (w;f_e) - f_e\|_{L^2(w)}^2 
    + \|s_{n-1} (\phi w;f_o) - f_o\|_{L^2(\phi w)}^2 \right).
$$
\end{cor}

\begin{proof}
By the displayed formulas at the end of the proof of the last theorem and 
$$
\int_0^1 |(1-x) g(x)|^2 w(x) dx = \int_0^1 |g(x)|^2 \phi w(x) dx = \|g\|_{L^2(\phi w)}^2,
$$
it is easy to see that 
\begin{align*}
 \|S_n f - f\|_{L^2(\Omega, w)}^2  = \, & \|s_n(w; f_e) - f_e+
       (1- \{\cdot\}) \left(s_{n-1} (\phi w;f_o) - f_o \right)\|_{L^2(w)}^2 \\
    & + \|s_n(w; f_e) - f_e - (1- \{\cdot\}) \left(s_{n-1} (\phi w;f_o) - f_o \right)\|_{L^2(w)}^2 \\
    = \, & 2 \left(\|s_{n} (w;f_e) - f_e\|_{L^2(w)}^2 
    + \|s_{n-1} (\phi w;f_o) - f_o\|_{L^2(\phi w)}^2 \right),
\end{align*}
where we have used the identity $(a+b)^2 + (a-b)^2 = 2 (a^2+b^2)$. 
\end{proof}

\subsection{Orthogonal structure on a wedge with Jacobi weight functions}

For $\a,\g > -1$, let $w_{\a,\g}$ be the Jacobi weight function defined by  
$$
w_{\a,\g}(x):= x^\a (1-x)^\g, \qquad x \in [0,1].
$$
We consider the inner product $\la \cdot,\cdot\ra_{w_1,w_2}$ defined in \eqref{eq:ipd-w} with
$w_1(x) = w_{\a,\g}(x)$ and $w_2(x) = w_{\b,\g}(x)$. More specifically, for $\a,\b,\g > -1$ and $\sigma > 0$, 
we define 
\begin{align*}
  \la f,g\ra_{\a,\b,\g}:=  c_{\a,\g} \int_0^1 f(x,1) g(x,1) w_{\a,\g}(x) dx  
           +\sigma  c_{\b,\g} \int_0^1 f(1,y) g(1,y) w_{\b,\g}(y) dy,  
\end{align*}
where 
$$
  c_{\a,\g} := \Big(\int_0^1w_{\a,\g}(x) dx \Big)^{-1} = \frac{\Gamma(\g+\a+2)}{\Gamma(\g+1)\Gamma(\a+1)}. 
$$

\subsubsection{Orthogonal structure}
We need to construct an explicit basis of $\CH_n^2(w_{\a,\g},w_{\b,\g})$. The case $\a = \b$ can be
regarded as a special case of Theorem \ref{thm:ipd-w-op}. The case $\a \ne \b$ is much more 
complicated, for which we need several properties of the Jacobi polynomials. 

Let $P_n^{(\a,\b)}$ denote the usual Jacobi polynomial of degree $n$ defined on $[-1,1]$. Then
$P_n^{(\g,\a)}(2x-1)$ is an orthogonal polynomial with respect to $w_{\a,\g}$ on $[0,1]$. Moreover,
\begin{align} \label{Jacobi-norm}
  h_n^{\a,\g} := &\, c_{\a,\g} \int_0^1 \left[P_n^{(\g,\a)}(2x-1)\right]^2 w_{\a,\g}(x) dx  \\
     = & \,  \frac{(\g+1)_n(\a+1)_n(n+\g+\a+1)}{ n! (\g+\a+2)_n(2n+\g+\a+1)}\notag
\end{align}
by \cite[(4.3.3)]{Szego}. Furthermore, $P_n^{(\a,\b)}(1) = \binom{n+\a}{n}$ and, in particular, $P_n^{(0,\b)}(1) =1$. 
Our construction relies on the following lemma. 

\begin{lem}
For $m > n \ge 0$, 
\begin{align*}
 I_{m,n}^{\a,\g}:= & c_{\a,\g}  \int_{0}^1 P_n^{(\g,\a)}(2x-1)  P_{m-1}^{(\g+2,\a)}(2x-1) (1-x)^{\g+1} x^\a dx  \\
      = &  \begin{cases} 0, & n > m, \\
      \displaystyle{ \frac{- m (\g+1)_m(\a+1)_m}{m!(2m+\g+\a+1) (\g+\a+2)_m},} & n = m, \\
      \\
       \displaystyle{\frac{(\g+1)(\a+1)_{m-1}(\g+1)_n}{(\g+\a+2)_{m}n!}}, & n < m.
         \end{cases}
\end{align*}
\end{lem}

\begin{proof}
Since $P_n^{(\g,\a)}(2x-1)$ is an orthogonal polynomial of degree $n$ with respect to $(1-x)^\g x^{\a}$ on 
$[0,1]$, $I_{m,n}^{\g,\a} = 0$ for $n > m$ holds trivially. For $m \ge n$, we need two identities of Jacobi polynomials. 
The first one is, see \cite[(4.5.4)]{Szego} or \cite[(18.9.6)]{DLMF},
$$
 (2m + \g+\a+1) (1-x)P_{m-1}^{(\g+2, \a)}(2x-1) = (m+\g+ 1) P_{m-1}^{(\g+1,\a)}(2x-1)- m P_m^{(\g+1,\a)}(2x-1)
$$
and the second one is the expansion, see \cite[(18.18.14)]{DLMF},
$$
  P_m^{(\g+1,\a)}(2x-1) = \frac{(\a+1)_m}{(\g+\a+2)_m}\sum_{k=0}^m \frac{(\g+\a+1)_k (2k+\g+\a+1)}{(\a+1)_k (\g+\a+1)} 
    P_k^{(\g,\a)}(2x-1).
$$
Putting them together shows that 
\begin{align}\label{eq:Jacobi-A}
 (1-x) P_{m-1}^{(\g+2,\a)}(2x-1)  = & \frac{(\g+1)(\a+1)_{m-1}}{(\g+\a+1)_{m+1}} \\
     & \times  \sum_{k=0}^{m-1}  \frac{(\g+\a+1)_k (2k+\g+\a+1)}{(\a+1)_k}P_k^{(\g,\a)}(2x-1) \notag \\
   & - \frac{m}{m+\g+\a+1}P_m^{(\g,\a)}(2x-1). \notag 
\end{align}
Substituting this expression into $I_{m,n}^{\g,\a}$ and using the orthogonality of the Jacobi polynomials and \eqref{Jacobi-norm}, 
we conclude that, for $m-1 \ge n$,  
$$
  I_{m,n}^{\a,\g} = \frac{(\g+1)(\a+1)_{m-1}}{(\g+\a+2)_{m}} \frac{ (\g+1)_n }{n!}.
$$
Hence, the case $m > n$ follows. The same argument works for the case $n =m$. 
\end{proof}

What is of interest for us is the fact that the dependence of $I_{m,n}^{\g,\a}$ on $n$ and $\a$ is separated, 
which is critical to prove that $Q_n$ in the next theorem is orthogonal.

\begin{thm} \label{thm:OP-ipd-abg}
Let $P_0(x,y) =1$ and, for $n =1,2,\ldots$, define 
\begin{align}
 P_n(x,y) = & \,P_n^{(\g,\a)}(2x-1) + P_n^{(\g,\b)}(2y-1) -  \binom{n+\g}{n}, \label{eq:op-P}\\ 
 Q_n(x,y) = & \, \frac{(\g+\a+2)_{n}}{(\a+1)_{n-1}} (1-x)P_{n-1}^{(\g+2,\a)}(2x-1) \label{eq:op-Q}\\
                    &  \qquad \qquad \qquad - \s^{-1} \frac{(\g+\b +2)_{n}}{(\b+1)_{n-1}} (1-y)P_{n-1}^{(\g+2,\b)}(2 y -1). \notag 
\end{align}
Then $\{P_n, Q_n\}$ are two polynomials in $\CH_n^2(w_{\a,\g},w_{\b,\g})$  and 
\begin{equation} \label{eq:ipdPQ}
  \la P_n, Q_n \ra_{\a,\b,\g} = \frac{(\b-\a) (\g+1)_{n+1}}{(2n+\g+\a+1)(2n+\g+\b+1) (n-1)!}.
\end{equation}
In particular, the two polynomials are orthogonal to each other if $\b = \a$. Furthermore
\begin{align*}
  \la P_n, P_n \ra_{\a,\b,\g} = &\, h_n^{\a,\g}+ \sigma h_n^{\b,\g} \\
  \la Q_n,Q_n\ra_{\a,\b,\g} = &\, 
  \frac{(\g+1)_2 (\a+\g+2)_n^2}{(\a+\g+2)_2 (\a+1)_{n-1}^2}   h_{n-1}^{\a,\g+2} +  \s^{-1}
         \frac{(\g+1)_2 (\b+\g+2)_n^2}{(\b+\g+2)_2 (\b+1)_{n-1}^2}  h_{n-1}^{\b,\g+2}  .
\end{align*}
\end{thm}

\begin{proof}
Since $P_n^{(\g,\a)}(1) = P_n^{(\g,\b)}(1) = \binom{n+\g}{n}$, our definition shows that
\begin{align*}
  \la P_n, Q_m \ra_{\a,\b,\g} = \frac{(\g+\a+2)_{m}}{(\a+1)_{m-1}} I_{m,n}^{\a,\g} 
     -\frac{(\g+\b+2)_{m}}{(\b+1)_{m-1}}  I_{m,n}^{\b,\g}. 
\end{align*}
By the identity in the previous lemma, if $n > m$, then $ \la P_n, Q_m \ra_{\a,\b,\g} =0$ since both
$I_{m,n}^{\a,\g}=0$ and $I_{m,n}^{\b,\g} =0$, whereas if $n < m$, then 
$$
 \la P_n, Q_m \ra_{\a,\b,\g} = \frac{(\g+1) (\g+1)_n}{n!} -\frac{(\g+1) (\g+1)_n}{n!}  =0.
$$
The case $n = m$ follows from a simple calculation. Moreover, for $m \ne n$, 
\begin{align*}
  \la P_n, P_m\ra_{\a,\b,\g} = & \, c_{\a,\g} \int_0^1 P_n^{(\g,\a)}(2x-1) P_m^{(\g,\a)}(2x-1)
         (1-x)^\g x^\a dx   \\
          &   + c_{\g,\g} \int_0^1 P_n^{(\b,\g)}(2x-1)  P_m^{(\g,\b)}(2x-1)(1-x)^\g x^\b dx  = 0 
\end{align*}
by the orthogonality of the Jacobi polynomials, and it is equal to $h_n^{\g,\a} + h_n^{\g,\b}$ for $m=n$. Similarly, 
\begin{align*}
 \la Q_n, Q_m & \ra_{\a,\b,\g} =  \frac{(\g+\a+2)_{m}}{(\a+1)_{m-1}} c_{\a,\g}
  \int_0^1 P_{n-1}^{(\g+2,\a)}(2x-1)  P_{m-1}^{(\g+2,\a)}(2x-1)(1-x)^{\g+2} x^\a dx \\
      & + \s^{-1} \frac{(\g+\b+2)_{m}}{(\b+1)_{m-1}}  c_{\b,\g}\int_0^1 P_{n-1}^{(\g+2,\b)}(2x-1)  P_{m-1}^{(\g+2,\b)}(2x-1)
        (1-x)^{\g+2} x^\b dx  = 0.
\end{align*}
To derive the norm of $\la Q_n,Q_n\ra$, we need to use $c_{\g,\a} = (\g+1)_2/(\a+\g+2)_2 c_{\g+2,\a}$. 
The proof is completed. 
\end{proof}

\begin{cor} \label{cor:OP-Rn}
For $n =1,2,\ldots$, define 
\begin{equation}\label{eq:op-R}
  R_n(x,y) = Q_n(x,y) -  \frac{\la P_n,Q_n\ra_{\a,\b,\g}}{h_n^{\g,\a}+ \s h_n^{\g,\b}} P_n(x,y).
\end{equation}
Then, for $\a \ne \b$, $\{P_n, R_n\}$ are two polynomials in $\CH_n^2(w_{\a,\b,\g})$ and they are mutually orthogonal. 
Moreover, 
$$
 \la R_n, R_n\ra_{\a,\b,\g} = \la Q_n,Q_n\ra_{\a,\b,\g} - \frac{\la P_n,Q_n\ra_{\a,\b,\g} }{\la P_n,P_n\ra_{\a,\b,\g}}. 
$$
\end{cor}

\subsubsection{Fourier orthogonal expansions}
Let $L^2(\Omega, w_{\a,\g}, w_{\b,\g})$ be the space of functions defined on $\Omega$ such that $f(1,1)$ is finite
and the norm 
$$
  \|f\|_{L^2(\Omega, w_{\a,\g}, w_{\b,\g})} = \left(c_{\a,\g} \int_0^1 |f(x,1)|^2 w_{\a,\g}(x) dx  
           + \s c_{\b,\g} \int_0^1 |f(1,y)|^2 w_{\b,\g}(y) dy \right)^{\f12}
$$
is finite for every $f$ in this space. For $ f\in L^2(\Omega, w_{\a,\g}, w_{\b,\g})$ we consider the Fourier orthogonal expansion 
with respect to $\la \cdot,\cdot\ra_{\a,\b,\g}$. With respect to the orthogonal basis $\{P_n,R_n\}$ in 
Theorem \ref{thm:OP-ipd-abg} and Corollary \ref{cor:OP-Rn}, the Fourier orthogonal expansion is defined by
$$
 f  =  \wh f_0 + \sum_{n=1}^\infty \left[  \wh f_{P_n}P_n + \wh f_{R_n}R_n  \right],
$$
where 
$$
 \wh f_0:= \frac{ \la f,1\ra_{\a,\b,\g}}{\la 1,1\ra_{\a,\b,\g}}, \qquad
 \wh f_{P_n}:= \frac{\la f,P_n \ra_{\a,\b,\g}}{\la P_n,P_n \ra_{\a,\b,\g}},
  \qquad \wh f_{R_n}:= \frac{\la f,R_n \ra_{\a,\b,\g}}{\la R_n,R_n \ra_{\a,\b,\g}}.
$$
Its $n$-th partial sum is defined by 
$$
  S_n^{\a,\b,\g} f: = \wh f_0 + \sum_{k =1}^n \left[  \wh f_{P_k}P_k + \wh f_{R_k}R_k  \right]. 
$$
In this case, we do not have a closed form for the reproducing kernel with respect to $\la \cdot,\cdot\ra_{\a,\b,\g}$. 
Nevertheless, we can relate the convergence of the Fourier orthogonal expansions to that of the Fourier--Jacobi series. 
For $w_{\a,\g}$, we denote the partial sum defined in \eqref{eq:partial-sum} by $s_n^{\a,\g} f$. 

For $f$ defined on $\Omega$, we define $f_1(x) = f(x,1)$ and $f_2(x) = f(1,x)$, and 
$$
  g_1(x): = \frac{f(x,1)-f(1,1)}{1-x} \quad \hbox{and}\quad g_2(y): = \frac{f(1,y)-f(1,1)}{1-y}. 
$$
It is easy to see that if $f(\cdot,1) \in L^2(w_{\a,\g},[0,1])$, then $g_1 \in L^2(w_{\a,\g+2}, [0,1])$, and if
$f(1,\cdot) \in L^2(w_{\b,\g},[0,1])$, then $g_2 \in L^2(w_{\b,\g+2}, [0,1])$. 

\begin{thm} \label{thm:3.4}
Let $\a, \b, \g > -1$. Then the Fourier orthogonal expansion converges in $f\in L^2(\Omega, w_{\a,\g}, w_{\b,\g})$. 
Furthermore, for $f(\cdot,1) \in L^2(w_{\a,\g})$ and $f(1,\cdot) \in L^2(w_{\b,\g})$, 
\begin{align*}
  \| f - S_n^{\a,\b,\g} f\|_{\a,\b,\g} \le &\, c \left( \|f_1 - s_n^{\a,\g} f_1 \|_{L^2(w_{\a,\g})} 
    + \| f_2 - s_n^{\b,\g} f_2\|_{L^2(w_{\b,\g})} \right) \\ 
      & +  c \left( \|g_1 - s_n^{\a,\g+2} g_1 \|_{L^2(w_{\a,\g+2})} 
       + \| g_2- s_n^{\b,\g+2} g_2\|_{L^2(w_{\b,\g+2})} \right),
\end{align*}
where $c$ is a constant that depends only on $\a,\b,\g$. 
\end{thm}

\begin{proof}
Since polynomials are dense on $\Omega$, by the Weierstrass theorem, the orthogonal basis $\{P_n,R_n\}$
is complete, so that the Fourier orthogonal expansion converges in $L^2(\Omega, w_{\a,\g}, w_{\b,\g})$. By 
the Parseval identity,
$$
   \| f - S_n^{\a,\b,\g} f\|_{\a,\b,\g}^2 = \sum_{k= n+1}^\infty |\wh f_{P_k}|^2 \la P_k, P_k\ra_{\a,\b,\g} +  
       \sum_{k= n+1}^\infty |\wh f_{R_k}|^2 \la R_k, R_k \ra_{\a,\b,\g}.
$$

Throughout this proof we use the convention $A \sim B$ if $c_1 B \le A \le c_2 A$, where $c_1$ and
$c_2$ are constants that are independent of varying parameters in $A$ and $B$. By \eqref{Jacobi-norm}
and the fact that $\Gamma(n+\a+1)/n! \sim n^\a$, it is easy to see that $h_n^{\a,\g} \sim n^{-1}$, so that 
$$
  \la P_n,P_n \ra_{\a,\b,\g} \sim n^{-1}, \qquad   \la Q_n,Q_n \ra_{\a,\b,\g} \sim n^{2\g+3}, \qquad
   \la P_n, Q_n\ra_{\a,\b,\g} \sim n^\g,
$$
and, consequently, 
$$
  \la R_n R_n\ra_{\a,\b,\g} \sim n^{2 \g+3} - n^{2 \g}/ n^{-1} \sim n^{2\g+3}. 
$$
The Fourier--Jacobi coefficients of $f_1$ and $f_2$ are denoted by $\wh{f_1}_n^{\a,\g}$ and
$\wh{f_2}_n^{\b,\g}$, respectively. It follows readily that $\wh f_{P_n} \sim  \wh{f_1}_n^{\a,\g} 
+ \wh{f_2}_n^{\b,\g}$, consequently,   
\begin{align*}
\sum_{k=n+1}^\infty |\wh f_{P_k}|^2 \la P_k,P_k\ra_{\a,\b,\g} \le &\, c \sum_{k=n+1}^\infty 
     \Big( |\wh{f_1}_k^{\a,\g}|^2 h_k^{\a,\g} + |\wh{f_2}_k^{\b,\g}|^2 h_k^{\b,\g}\Big) \\
     \le &\, c  \left( \|f_1 - s_n^{\a,\g} f_1 \|_{L^2(w_{\a,\g})} 
    + \| f_2 - s_n^{\b,\g} f_2\|_{L^2(w_{\b,\g})} \right).     
\end{align*}
We now consider the estimate for $R_n$ part. By the definition of $R_n$, 
\begin{align*}
   \la f, R_n\ra_{\a,\b,\g} \sim \la f, Q_n \ra_{\a,\b,\g} - n^{\g+1} \la f, P_n \ra_{\a,\b,\g}. 
\end{align*}
It is easy to see that 
$$
  \sum_{k=n+1}^\infty \frac{ | k^{\g+1}\la f, P_k \ra_{\a,\b,\g} |^2}{\la R_k,R_k\ra_{\a,\b,\g}} \sim \sum_{k=n+1}^\infty 
       k^{-1}  | \la f, P_k \ra_{\a,\b,\g} |^2 \sim 
   \sum_{k=n+1}^\infty k^{-2} |\wh f_{P_k}|^2 \la P_k,P_k\ra_{\a,\b,\g}, 
$$
so that we only have to work with the term $\la f, Q_k \ra_{\a,\b,\g}$. The definition of $Q_k$ shows that 
$\la 1, Q_k\ra_{\a,\b,\g} = 0$, which leads to the identity 
\begin{align*}
  \la f, Q_k \ra_{\a,\b,\g} = &\, \frac{(\g+\a+2)_k}{(\a+k)_{k-1}} c_{\a,\g}\int_0^1 (f(x,1) - f(1,1)) Q_k(x,1) x^\a (1-x)^\g dx \\
  & +  \frac{(\g+\b+2)_k}{(\b+n)_{k-1}}  c_{\b,\g}\int_0^1 (f(1,y) - f(1,1)) Q_k(1,y) y^\b (1-y)^\g dy\\
    = &\,\frac{(\g+\a+2)_k}{(\a+k)_{k-1}} \wh {g_1}_k^{\a,\g+2} h_k^{\a,\g+2} + 
  \frac{(\g+\b+2)_k}{(\b+n)_{k-1}}  \wh {g_2}_k^{\b,\g+2} h_k^{\b,\g+2}. 
\end{align*}
Consequently, it follows that 
\begin{align*}
 \sum_{k=n+1}^\infty \frac{|\la f, Q_k \ra_{\a,\b,\g} |^2}{\la R_k,R_k\ra_{\a,\b,\g}} 
 & \le c  \sum_{k=n+1}^\infty \left( k |\wh {g_1}_k^{\a,\g+2} h_k^{\a,\g+2}|^2 
     + k |\wh {g_2}_k^{\b,\g+2} h_k^{\b,\g+2}|^2 \right) \\
 & \le c  \sum_{k=n+1}^\infty \left( |\wh {g_1}_k^{\a,\g+2}|^2 h_k^{\a,\g+2}+ |\wh {g_2}_k^{\b,\g+2}|^2 h_k^{\b,\g+2}\right) \\
  & =  c \left( \|g_1 - s_n^{\a,\g+2} g_1 \|_{L^2(w_{\a,\g+2})}  + \| g_2- s_n^{\b,\g+2} g_2\|_{L^2(w_{\b,\g+2})} \right).
\end{align*} 
The proof is completed. 
\end{proof}

\section{Orthogonal polynomials on the boundary of the square}\label{Section:BoundarySquare}
\setcounter{equation}{0}

Using the results in the previous section, we can study orthogonal polynomials on a parallelogram. Since orthogonal 
structure is preserved under an affine transformation, we can assume without loss of generality that the parallelogram
is the square $[-1,1]^2$. 

For $\a,\g > -1$, let $\varpi_{\a,\g}$ be the weight function
$$
 \varpi_{\a,\g}(x):= |x|^{2\a+1} (1-x^2)^\g. 
$$
We consider orthogonal polynomials of two variables on the boundary of $[-1,1]^2$ with respect to the
bilinear form
\begin{align} \label{eq:int-boundary}
  \la f, g \ra = & \,  c_{\a,\g}\int_{-1}^1 [f(x,-1) g(x,-1) + f(x,1) g(x,1)]  \varpi_{\a,\g}(x) dx  \\ 
        &+   c_{\b,\g} \int_{-1}^1 [f(-1,y) g(-1,y)+ f(1,y) g(1,y)] \varpi_{\b,\g}(y) dy  \notag
\end{align}
for $\a,\b,\g > -1$. Since $(1-x^2)(1-y^2)$ vanishes on the boundary of the square, the bilinear form defines an 
inner product modulo the ideal generated by this polynomial, or in the space  
$$
  \RR[x,y]/I : = \RR[x,y] /\la (1-x^2) (1-y^2) \ra.
$$ 
Let $\CB\CV_n^2$ denote the space of orthogonal polynomials in $\RR[x,y]/I$ with respect to the inner product 
$\la \cdot,\cdot\ra$. 

\begin{prop}
For $n \ge 0$, the dimension of $\CB\CV_n^2$ is given by 
$$
  \dim \CB\CV_n^2 = n+1, \quad n =0,1,2, \quad \hbox{and}  \quad \dim \CB\CV_n^2 =4, \quad  n \ge 3.
$$
\end{prop}

Recall that the inner product $\la \cdot,\cdot\ra_{\a,\b,\g}$ studied in the previous section contains a fixed parameter
$\sigma$. For fixed $\a,\b$ and $\d_1,\d_2 \in \{0,1\}$, we define $p_{m,1}^{\a+\d_1,\b+\d_2,\g}$ and 
$p_{m,2}^{\a+\d_1,\b+\d_2,\g}$ to be a basis of $\CH_m^2(w_{\a+\d_1,\g},w_{\b+\d_2,\g})$ for a particular choice of 
$\s$ defined by
\begin{equation} \label{eq:sigma}
    \s_{\d_1,\d_2} =  \frac{ c_{\b,\g} c_{\a+\d_1,\g}}{c_{\a,\g} c_{\b+\d_2,\g}}.
\end{equation}
For example, $p_{m,i}^{\a,\b,\g}$ are defined with $\s_{0,0} =1$ and $p_{m,i}^{\a+1,\b,\g}$ are defined with 
$\s_{1,1}  = (\a+\g+2)/(\a+1)$. For each pair of $\a+\d_1, \, \b+\d_2$, we can choose, for example, 
$p_{m,1}^{\a+\d_1,\b+\d_2,\g}  = P_m$ defined in \eqref{eq:op-P} and  take $p_{m,2}^{\a+\d_1,\b+\d_2,\g} 
= Q_m$ defined in \eqref{eq:op-Q} or $p_{m,2}^{\a+1,\b+1,\g} = R_m$ defined in \eqref{eq:op-R}. 

\begin{thm} \label{thm:boundaryOP}
For $n = 0, 1,2,$ a basis for $\CB\CV_n$ is denoted by $Y_{n,i}$ and given by 
\begin{align*}
  &  Y_{0,1}(x,y) = 1, \quad Y_{1,1}(x,y) = x \quad Y_{1,2}(x,y) = y,  \\ 
  & Y_{2,1}(x,y) = p_{1,1}^{\a,\b,\g}(x^2,y^2),\quad Y_{2,2}(x,y) = xy, \quad Y_{2,3}(x,y) = p_{1,2}^{\a,\b,\g}(x^2,y^2).
\end{align*}
For $n \ge 3$, the four polynomials in $\CB\CV_n^2$ that are linearly independent modulo the 
ideal can be given by 
\begin{align*}
  Y_{2m,1}(x,y) & =  p_{m,1}^{\a,\b,\g}(x^2,y^2),\\
  Y_{2m,2}(x,y) & =  p_{m,2}^{\a,\b,\g}(x^2,y^2),\\
  Y_{2m,3}(x,y) & =  x y \,p_{m-1,1}^{\a+1,\b+1,\g}(x^2,y^2),\\
  Y_{2m,4}(x,y) & =  x y \, p_{m-1,2}^{\a+1,\b+1,\g}(x^2,y^2)
\end{align*}
for $n =2m \ge 2$, and
\begin{align*}
  Y_{2m+1,1}(x,y) & =  x \,p_{m,1}^{\a+1,\b,\g}(x^2,y^2),\\
  Y_{2m+1,2}(x,y) & =  x \,p_{m,2}^{\a+1,\b,\g}(x^2,y^2),\\
  Y_{2m+1,3}(x,y) & =  y \,p_{m,1}^{\a,\b+1,\g}(x^2,y^2), \\
  Y_{2m+1,4}(x,y) & =  y \,p_{m,2}^{\a,\b+1,\g}(x^2,y^2)
\end{align*}
for $n=2m+1 \ge 3$. In particular, these bases satisfy the equation $\partial_x^2 \partial_y^2 u = 0$. 
\end{thm}
  
\begin{proof}
The proof relies on the parity of the integrals. For example, it is easy to see that $\la x f(x^2,y^2), g(x^2,y^2) \ra =0$ 
and $\la y f(x^2,y^2), g(x^2,y^2) \ra =0$ for any polynomials $f$ and $g$, which implies, in particular, that 
$\la Y_{2m,i}, Y_{2n+1,j}\ra =0$ for  $i,j = 1,2,3,4$. Furthermore, it is easy to see that $\la x y f(x^2,y^2), g(x^2,y^2)\ra = 0$ for any polynomials $f$ 
and $g$. Hence, $\la Y_{2m,i},Y_{2k,j}\ra =0$ for $i =1,2$ and $j=3,4$. 
Furthermore, using the relation
\begin{equation} \label{eq:4.3}
   \int_{-1}^1 f(x^2) |x|^{2\a+1} (1-x^2)^\g dx =  \int_0^1 f(x) x^\a (1-x)^\g dx,
\end{equation}
it is easy to see that
\begin{align*}
   \la Y_{2m,i}, Y_{2 k,j} \ra & = \la p_{m,i}^{\a, \b,\g}, p_{k,j}^{\a,\b,\g} \ra_{\a,\b,\g}, \quad i,j = 1, 2 \\
   \la Y_{2m,i}, Y_{2 k,j} \ra & = \frac{c_{\a,\g}}{c_{\a+1,\g}} \la p_{m,i}^{\a+1, \b+1,\g}, 
      p_{k,j}^{\a+1, \b+1,\g} \ra_{{\a+1, \b+1,\g}},  \quad i,j = 3,4,
\end{align*}
where in the second identity, we have adjusted the normalization constants of integrals from $c_{\a,\g}$ and 
$c_{\b,\g}$ to $c_{\a+1,\g}$ and $c_{\b+1,\g}$, respectively, and used our choice of $\s_{1,1}$. Hence,
with our choice of $\s_{0,0}$ and $\s_{1,1}$, we see that $Y_{2m,i}$ is orthogonal to $Y_{2k,j}$ for 
$i,j =1,2$ and $i,j = 3,4$, respectively. Similarly, by the same consideration, we obtain that 
\begin{align*}
   \la Y_{2m+1,i}, Y_{2 k+1,j} \ra &\, = \frac{c_{\a,\g}}{c_{\a+1,\g}} 
        \la p_{m,i}^{\a+1,\b,\g}, p_{k,j}^{\a+1,\b,\g} \ra_{\a+1, \b,\g}, \quad i,j = 1, 2 \\
   \la Y_{2m+1,i}, Y_{2 k+1,j} \ra & \, = \la p_{m,i}^{\a,\b+1,\g}, p_{k,j}^{\a,\b+1,\g} \ra_{\a,\b+1,\g},  \quad i,j = 3,4,
\end{align*}
which shows, with our choice of $\s_{0,1}$ and $\s_{1,0}$, that $Y_{2m+1,i}$ is orthogonal to $Y_{2k+1,j}$ for 
$i,j =1,2$ and $i,j = 3,4$, respectively. Finally, since $\partial_x \partial_y p_{n,i}^{\a,\b}(x,y) =0$, we see that 
$Y_{n,j} = \xi(x,y) u(x) + \eta(x,y) v(x)$, where $\xi$ and $\eta$ are linear polynomial of $x,y$, so that it is evident
that $\partial_x^2 \partial_y^2 Y_{n,j}(x,y)=0$.  
\end{proof} 

In our notation, the case $\a = -\frac12$ $\b= - \f12$ and $\g =0$ corresponds to the inner product in which 
the integrals are unweighted. 

Let $L^2([-1,1]^2, \varpi_{\a,\g},  \varpi_{\b,\g})$ be the space of functions defined on the boundary of $[-1,1]^2$ such
that $f(\pm 1, \pm 1)$ are finite and the norm
\begin{align*}
  \|f\|_{L^2(\varpi_{\a,\g},  \varpi_{\b,\g})} = & 
     \left( c_{\a,\g} \int_{-1}^1 \left( |f(x,1)|^2+|f(x,-1)|^2 \right) \varpi_{\a,\g}(x)dx \right. \\ 
    &  \left.  + c_{\b,\g} \int_{-1}^1 \left( |f(1,y)|^2+|f(-1,y)|^2 \right) \varpi_{\b,\g}(y)dy \right)^{\f12}.
\end{align*}
is finite for every $f$. For $f \in L^2([-1,1]^2, \varpi_{\a,\g},  \varpi_{\b,\g})$, its Fourier orthogonal expansion is defined by
$$
  f = \sum_{n=0}^2 \sum_{i=1}^{n+1} \wh f_{n,i} Y_{n,i}^{\a,\b,\g} +  \sum_{n=3}^\infty
       \sum_{i=1}^{4} \wh f_{n,i} Y_{n,i}^{\a,\b,\g},
        \qquad \wh f_{n,i} = \frac{\la f, Y_{n,i}^{\a,\b,\g}\ra}{\la Y_{n,i}, Y_{n,i}^{\a,\b,\g}\ra}. 
$$
For $n \ge 2$, let $S_n (f)$ denotes its $n$-th partial sum defined by  
$$
  S_n f = \sum_{k=0}^2 \sum_{i=1}^{k+1} \wh f_{k,i} Y_{k,i}^{\a,\b,\g} +  \sum_{k=3}^n
       \sum_{i=1}^{4} \wh f_{k,i} Y_{k,i}^{\a,\b,\g}. 
$$ 

For fixed $\a,\b,\g$, let $\la \cdot,\cdot\ra_{\a+\delta_1,\b+\delta_2,\g}$ be the inner product defined in the previous 
section with $\s = \s^{\a,\b,\g}$. For $f$ defined on $[-1,1]^2$, we define four functions
\begin{align*}
  F_{e,e}(x,y) &\, = \tfrac14\left[ f(x,y) + f(-x,y)+ f(x,-y) + f(-x,-y) \right], \\
  F_{e,o}(x,y) &\, = \tfrac14\left[ f(x,y) + f(-x,y) - f(x,-y) - f(-x,-y) \right], \\
  F_{o,e}(x,y) &\, = \tfrac14\left[ f(x,y) - f(-x,y) + f(x,-y) - f(-x,-y) \right], \\
  F_{o,o}(x,y) &\, = \tfrac14\left[ f(x,y) - f(-x,y) - f(x,-y) + f(-x,-y) \right], 
\end{align*}
where the subindices indicate the parity of the function. For example, $F_{e,o}$ is even in $x$ variable and odd in $y$ variable. 
By definition,
$$
   f(x,y) = F_{e,e}(x,y) + F_{e,o}(x,y) + F_{o,e}(x,y) +F_{o,o}(x,y).
$$
We further define 
\begin{align*}
   G_{0,0}(x,y) &\, = F_{e,e}(x,y), \quad   G_{0,1}(x,y) = y^{-1} F_{e,o}(x,y) , \\
   G_{1,0}(x,y) & \, = x^{-1} F_{o,e}(x,y), \quad   G_{1,1}(x,y) = x^{-1} y^{-1} F_{o,o}(x,y)
\end{align*}
and define $\psi: \RR^2 \mapsto \RR^2$ by $\psi: (x,y) \mapsto (\sqrt{x},\sqrt{y})$. Changing variables in integrals as
in \eqref{eq:4.3}, we see that if $f\in L^2([-1,1]^2, \varpi_{\a,\g}, \varpi_{\b,\g})$, then
$G_{\delta_1,\delta_2}\circ \psi \in L^2(\CB, w_{\a+\delta_1,\g}, w_{\b+\delta_2,\g})$ for $\delta_i \in \{0,1\}$. 

\begin{thm}\label{thm:squareconv}
For $f\in L^2([-1,1]^2, \varpi_{\a,\g}, \varpi_{\b,\g})$, 
\begin{align*}
S_{2 m} f (x,y) = &\, S_m^{\a,\b,\g} G_{0,0}\circ \psi (x^2,y^2)+ y S_{m-1}^{\a,\b+1,\g} G_{0,1} \circ \psi (x^2,y^2) \\
     &\,  + x S_{m-1}^{\a+1,\b,\g} G_{1,0} \circ \psi(x^2,y^2)
    + x y S_{m-1 }^{\a+1,\b+1,\g} G_{1,1} \circ \psi(x^2,y^2),  
\end{align*}
\begin{align*}
S_{2 m+1} f (x,y) = &\, S_m^{\a,\b,\g} G_{0,0}\circ \psi (x^2,y^2)+ y S_{m}^{\a,\b+1,\g} G_{0,1} \circ \psi (x^2,y^2) \\
     &\,  + x S_{m}^{\a+1,\b,\g} G_{1,0} \circ \psi(x^2,y^2)
    + x y S_{m-1}^{\a+1,\b+1,\g} G_{1,1} \circ \psi(x^2,y^2). 
\end{align*}
In particular, the norm of $S_n f - f$ is bounded by those of $S_{m}^{\a+\d_1,\b+\d_2,\g} G_{\d_1,\d_2} - G_{\d_1,\d_2}$
as in Theorem \ref{thm:3.4}.
\end{thm}

\begin{proof}
Using the parity of the function, it is easy to see that 
$$
  \frac{\la f, Y_{2m,i}\ra}{\la Y_{2m,i}, Y_{2m,i}\ra} = \frac{\la F_{e,e}, Y_{2m,i} \ra}{\la Y_{2m,i}, Y_{2m,i}\ra}  =
   \frac{\la G_{0,0} \circ \psi, p_{2m,i}^{\a,\b,\g}\ra_{\a,\b,\g}}{\la  p_{2m,i}^{\a,\b,\g}, p_{2m,i}^{\a,\b,\g}\ra_{\a,\b,\g}}, \qquad i = 1,2,
$$
where we have used the fact that $F_{e,e}$ is even in both variables and use the change of variables in integrals 
as in \eqref{eq:4.3}. The similar procedure can be used in the other three cases, as $G_i(x,y)$ is even in both
variables, and the result is 
\begin{align*}
  \frac{\la f, Y_{2m,i}\ra}{\la Y_{2m,i}, Y_{2m,i}\ra} &\, = \frac{\la F_{o,o}, Y_{2m,i} \ra}{\la Y_{2m,i}, Y_{2m,i}\ra}  =
  \frac{\la G_{1,1} \circ \psi, p_{m,i}^{\a+1,\b+1,\g}\ra_{\a+1,\b+1,\g}}{\la  p_{2m,i}^{\a+1,\b+1,\g}, p_{2m,i}^{\a+1,\b+1,\g}
    \ra_{\a+1,\b+1,\g}}, \quad i = 3,4, \\
  \frac{\la f, Y_{2m+1,i}\ra}{\la Y_{2m+1,i}, Y_{2m+1,i}\ra} &\, = \frac{\la F_{e,o}, Y_{2m+1,i} \ra}{\la Y_{2m+1,i}, Y_{2m+1,i}\ra}  =
   \frac{\la G_{0,1} \circ \psi, p_{m,i}^{\a,\b+1,\g}\ra_{\a,\b+1,\g}}{\la  p_{2m,i}^{\a,\b+1,\g}, p_{2m,i}^{\a,\b+1,\g}
    \ra_{\a,\b+1,\g}}, \quad i = 1,2, \\
         \frac{\la f, Y_{2m+1,i}\ra}{\la Y_{2m+1,i}, Y_{2m+1,i}\ra} &\, = \frac{\la F_{o,e}, Y_{2m+1,i} \ra}{\la Y_{2m+1,i}, Y_{2m+1,i}\ra}  =
  \frac{\la G_{1,0} \circ \psi, p_{m,i}^{\a+1,\b,\g}\ra_{\a+1,\b,\g}}{\la  p_{2m,i}^{\a+1,\b,\g}, p_{2m,i}^{\a+1,\b,\g}
    \ra_{\a+1,\b,\g}}, \quad i = 3,4. 
\end{align*}
Since $S_n^{\a+\d_1, \b+ \d_2, \g} G_{\d_1,\d_2} \circ \psi(x^2, y^2) \to G_{\d_1,\d_2} (x,y)$ and 
$$
        f(x,y) = G_{0,0}(x,y) + y G_{0,1}(x,y) + y G_{1,0}(x,y) + xy G_{1,1}(x,y),
$$
the last statement is evident. 
\end{proof}

\section{Orthogonal system on the square}\label{Section:Square}
\setcounter{equation}{0}

Let $w$ be a nonnegative weight function defined on $[0,1]$. Define 
$$
   W(x,y) = w(\max \{|x|,|y|\}), \qquad (x,y) \in [-1,1]^2.
$$  
We construct a system of orthogonal functions with respect to the inner product
$$
  \la f,\g\ra_W = \int_{-1}^1 \int_{-1}^1 f(x,y) g(x,y) W(x,y) dx dy. 
$$
by making use of the orthogonal polynomials on the boundary or the square, studied in the previous section. 
Our starting point is the following integral identity derived from changing variables $(x,y)
 \mapsto (s \xi,s\eta)$, 
\begin{equation} \label{eq:int-square}
  \int_{-1}^1 \int_{-1}^1 f(x,y) w(\max \{|x|,|y|\})dxdy = \int_0^s s \int_\CB f(s \xi, s \eta) d\s(\xi,\eta) w(s) ds,
\end{equation}
where $\int_\CB d \sigma$ denotes the integral on the boundary of the square,
$$
  \int_\CB f(\xi,\eta)d\s(\xi,\eta) = \int_{-1}^1 \left[f(\xi,1)+ f(\xi,-1)\right] d\xi +  \int_{-1}^1 \left[f(1 \eta)+ f(-1, \eta)\right] d\eta.
$$

Our orthogonal functions are similar in structure to orthogonal polynomials on the unit disk that are constructed 
using spherical harmonics. However, these function are polynomials in $(s, \xi,\eta)$ for the $(x,y) = (s \xi, s\eta) 
\in [-1,1]^2$, but not polynomials in $(x,y)$. 

Let $\CB\CV_n^2$ be the space of orthogonal polynomials on the boundary of $[-1,1]^2$ with respect to the inner
product
$$
 \la f,g\ra = \int_\CB f(\xi,\eta) g(\xi,\eta) d\s(\xi,\eta) , 
$$
which is the inner product with $\a = -\f12$, $\b = - \f12$ and $\g =0$ studied in the previous section. Let $Y_{n,i}$ be an 
orthogonal basis for $\CB\CV_n^2$. For $n \le 2$, they are defined by, see Theorem \ref{thm:boundaryOP}, 
\begin{align*}
 Y_{0,1}(x,y) & =1, \quad Y_{1,1}(x,y) = x, \quad Y_{1,2}(x,y) = y;  \\
 Y_{2,1}(x,y) & =  x^2 - \frac{2}{3}, \quad Y_{2,2}(x,y) = xy,\quad
 Y_{2,3}^2(x,y) = y^2 - \frac{2}{3},
\end{align*}
whereas for $n \ge 3$, they are constructed in Theorem \ref{thm:boundaryOP}. For $n \ge k $, denote 
by $P_{m, 2n-2k}$ the orthogonal polynomial of degree $m$ with respect to $t^{2n-2k+1} w(t)$ on $[0,1]$
and with $P_{0,2n-2k}(s):=1$. For $n \ge 0$ and $0 \le k \le n$, we define
$$
   Q_{k,i}^n (x,y):= P_{k, 2n-2k}(s) s^{n-k} Y_{n-k,i}\left(\frac{\xi}{s}, \frac{\eta}{s} \right),  
$$
where $i = 1,\ldots, \min\{n+1,4\}$. 

\begin{thm}
In the coordinates $(x,y) = s(\xi,\eta)$, the system of functions 
$$
      \{Q_{k,i}^n: i = 1,\ldots, \min\{n+1,4\}, \,\,  0\le k\le n, \,\, n\ge 0\}
$$
is a complete orthogonal basis for $L^2(W; [-1,1]^2)$.  
\end{thm}

\begin{proof}
Changing variables $x = s \xi$ and $y= \eta$ shows 
\begin{align*}
 \la Q_{k,i}^n, Q_{l,j}^m \ra_{W} = & \int_0^1 P_{k, 2n-2k}(s)P_{l, 2m-2l}(s) s^{n-k+m-l+1} w(s)ds \\
    & \times \int_{\CB} Y_{n- k,i}(\xi,\eta) Y_{m- l,j}(\xi,\eta) d\s(\xi,\eta).
\end{align*}
The second integral is zero if $i \ne j$ and $n-k  \ne m-l$, whereas the second integral is zero when
$n-k = m-l$ and $k \ne l$, so that $\la Q_{k,i}^n, Q_{l,j}^m \ra_{W} =0$ if $ i\ne j$, $k\ne l$ and $n \ne m$. 
By definition, $s^{n-k} Y_{n-k,i}\left(\frac{\xi}{s}, \frac{\eta}{s} \right)$ is a polynomial of degree $n-k$ in 
the variable $s$, so that $Q_{k,i}^n$ is a polynomial of degree $n$. To show that the system is complete,
we show that if $\la f, Q_{k,i}^n\ra =0$ for all $k,i,n$, then $f(x,y)=0$. Indeed, by the orthogonality of 
polynomials on the boundary, we see that
\begin{align*}
  f (x,y) = f(s \xi,s \eta) & = \sum_{k=0}^n s^k \sum_{j=0}^k a_{j,k} \xi^j \eta^{k-j} \\
   & =  \sum_{k=0}^n s^k   \sum_{m=0}^k \sum_{i=1}^{\min\{m+1,4\}} b_{m,i}^k Y_{m,i}(\xi,\eta)
\end{align*}
modulo the ideal. Changing order of summation shows that 
\begin{align*}
 f(x,y) = & \sum_{m=0}^n \sum_{i=1}^{\min\{m+1,4\}}  \left(\sum_{k=0}^{n-m} b_{m,i}^{k+m} s^k \right) s^m Y_{m,i}(\xi,\eta)\\
  & = \sum_{m=0}^n \sum_{i=1}^{\min\{m+1,4\}}  \left(\sum_{k=0}^{n-m} c_{m,i,k} P_{k,2m}(x) \right) s^m Y_{m,i}(\xi,\eta).
\end{align*}
This completes the proof. 
\end{proof}

\section{Sampling the associated determinantal point process}\label{sec:detpointprocess}

Associated with an orthonormal basis $q_0(x),\ldots,q_N(x)$ is a determinantal point process, which describes $N$ points $\lambda_1,\ldots,\lambda_N$ distributed according to
$$
	\det \begin{pmatrix} K_N(\lambda_1,\lambda_1) & \cdots & K_N(\lambda_1, \lambda_N) \cr
					\vdots & \ddots & \vdots \cr
					K_N(\lambda_N,\lambda_1) & \cdots & K_N(\lambda_N,\lambda_N)
					\end{pmatrix}
					$$
where 
$$
	K_N(x,y) = \sum_{k=0}^N q_k(x)q_k(y)
$$
is the reproducing kernel,  see \cite{RMTHandbookDetProcess} for an overview of determinantal point processes. 

In the particular case of univariate orthogonal polynomials with respect to a weight $w(x)$, the associated determinantal process is equivalent to  a Coulomb gas---that is, the points are distributed according to 
	$${1 \over Z_N} \prod_{k=1}^N w(x_k) \prod_{k < j} |\lambda_k - \lambda_j|^2 $$
where $Z_N$ is the normalization constant---as well as the eigenvalues of unitary ensembles, see for example \cite{DeiftOrthogonalPolynomials} for the case of an analytic weight on the real line or \cite{JUEUniversality} for the case of a weight supported on $[-1,1]$ with Jacobi-like singularities.

 In the case of our orthogonal polynomials on the wedge, the connection with Coulomb gases and random matrix theory is no longer obvious: the interaction of the points is not Coulomb (that is, it can not be reduced to a Vandermonde determinant squared times a product of weights), nor is there an obvious distribution of random matrices whose eigenvalues are associated with the points\footnote{If there is such a random matrix distribution, one would expect it to be a pair of commuting random matrices, whose joint eigenvalues give points on the wedge.}. We note that there are recent universality results due to Kro\'o and Lubinsky on the asymptotics of Christoffel functions associated with multivariate orthogonal polynomials   \cite{UniversalityBall,UniversalityMultivariateOPs}, but they do not apply in our setting.

\begin{figure}
 \begin{center}
 \includegraphics[width=.7\textwidth]{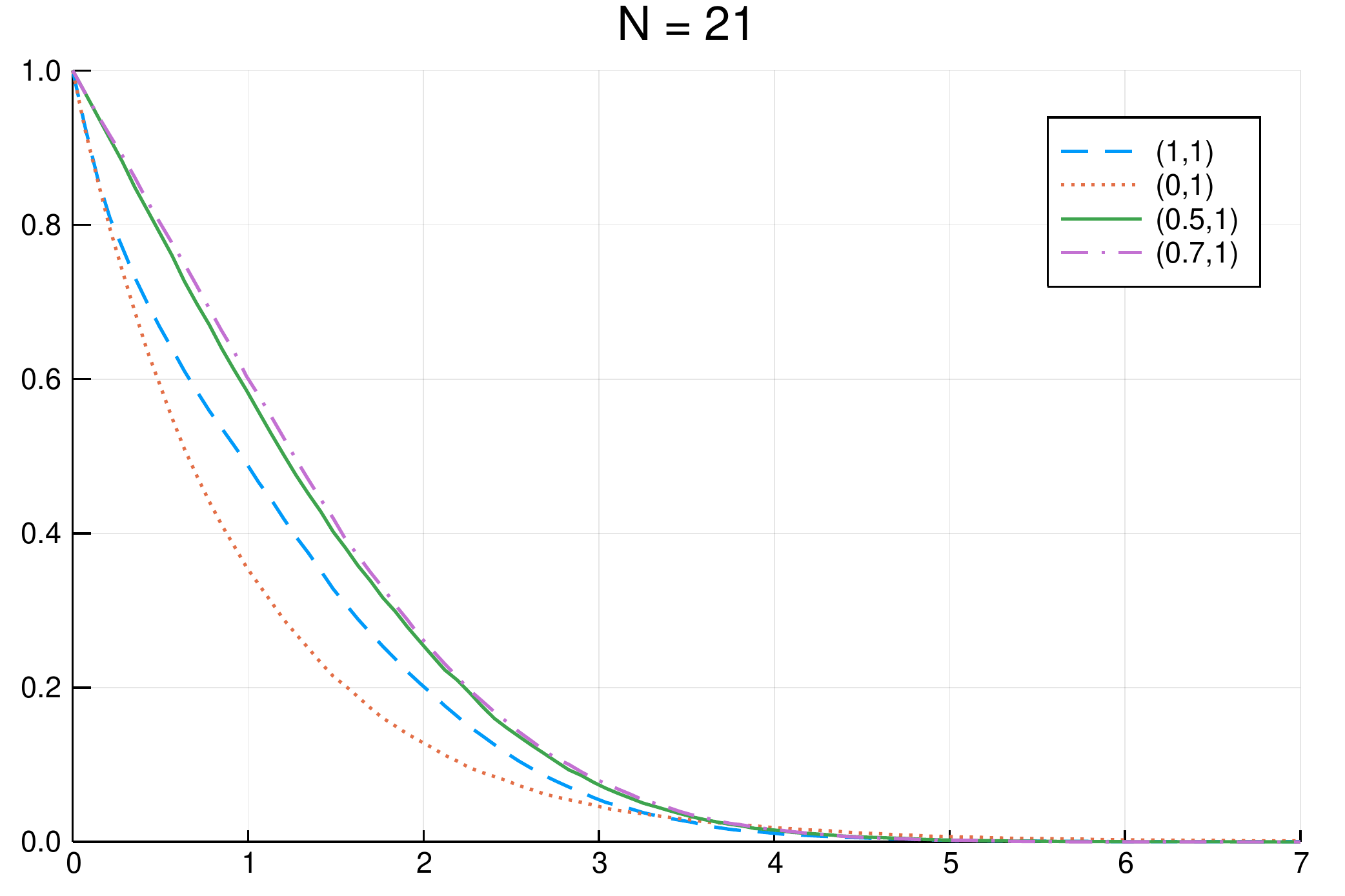}
  \includegraphics[width=.7\textwidth]{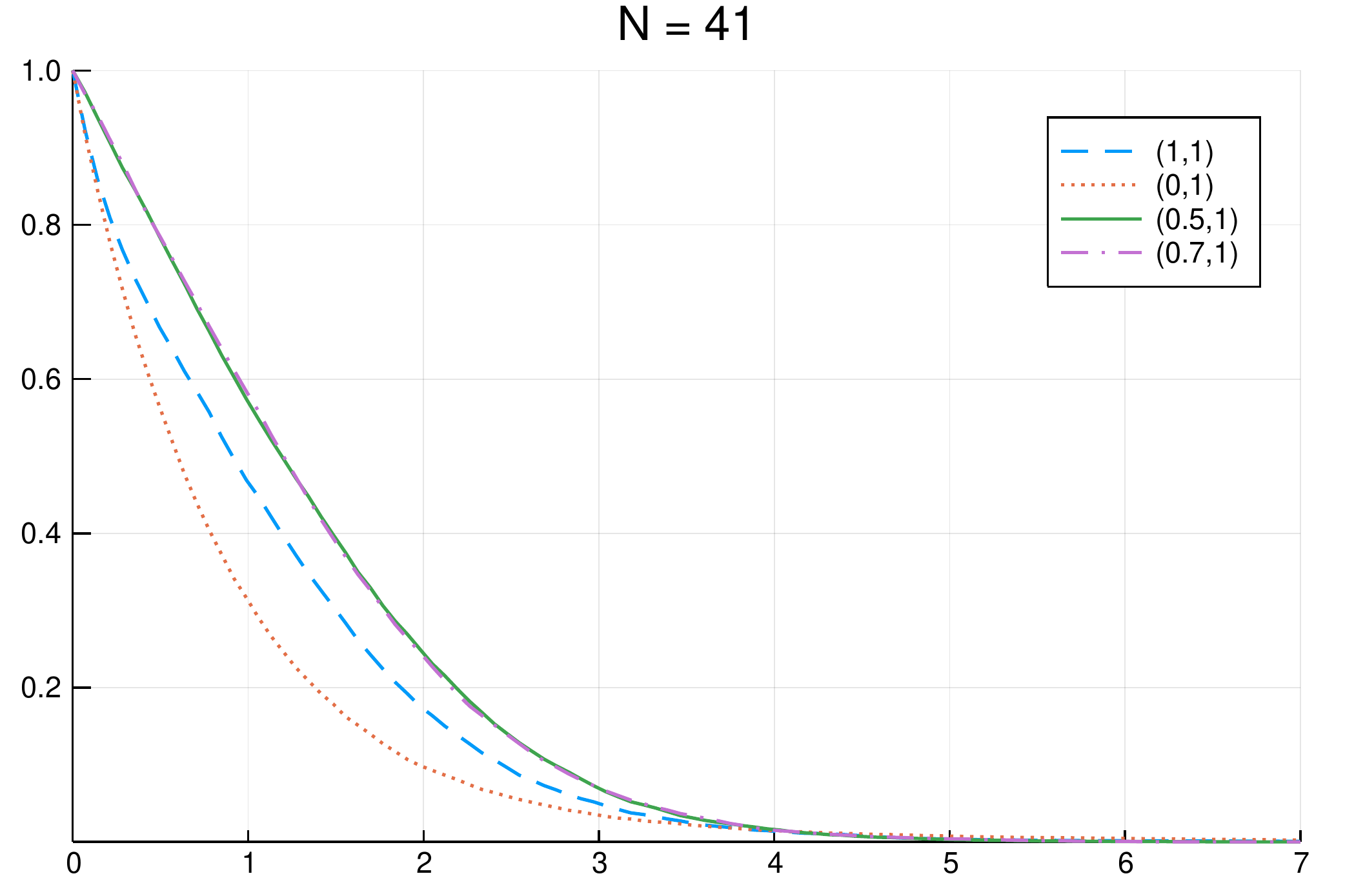}
   \includegraphics[width=.7\textwidth]{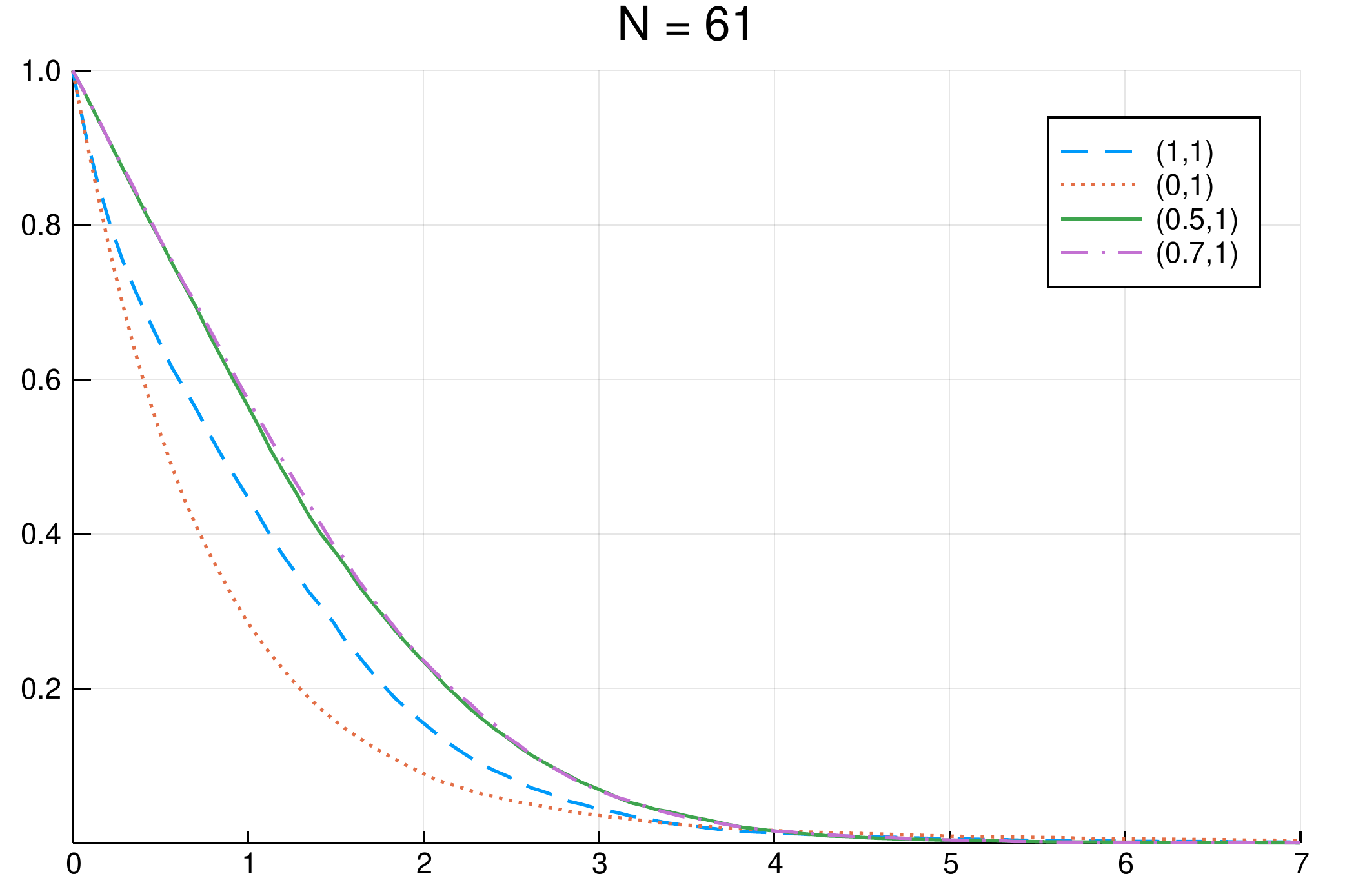}
  \caption{Monte Carlo calculation of the probability that no point satisfying $y = 1$ sampled according to the determinantal point process associated to the wedge orthogonal polynomials with $\alpha = \beta = \gamma = 0$ lies in a neighbourhood of four different points.  $N$ is the total number of basis elements and points.  We have scaled the statistics so that the variance is one, and have used 10,000 samples. }\label{fig:detprocess}
 \end{center}
\end{figure}

\begin{figure}
 \begin{center}
 \includegraphics[width=.7\textwidth]{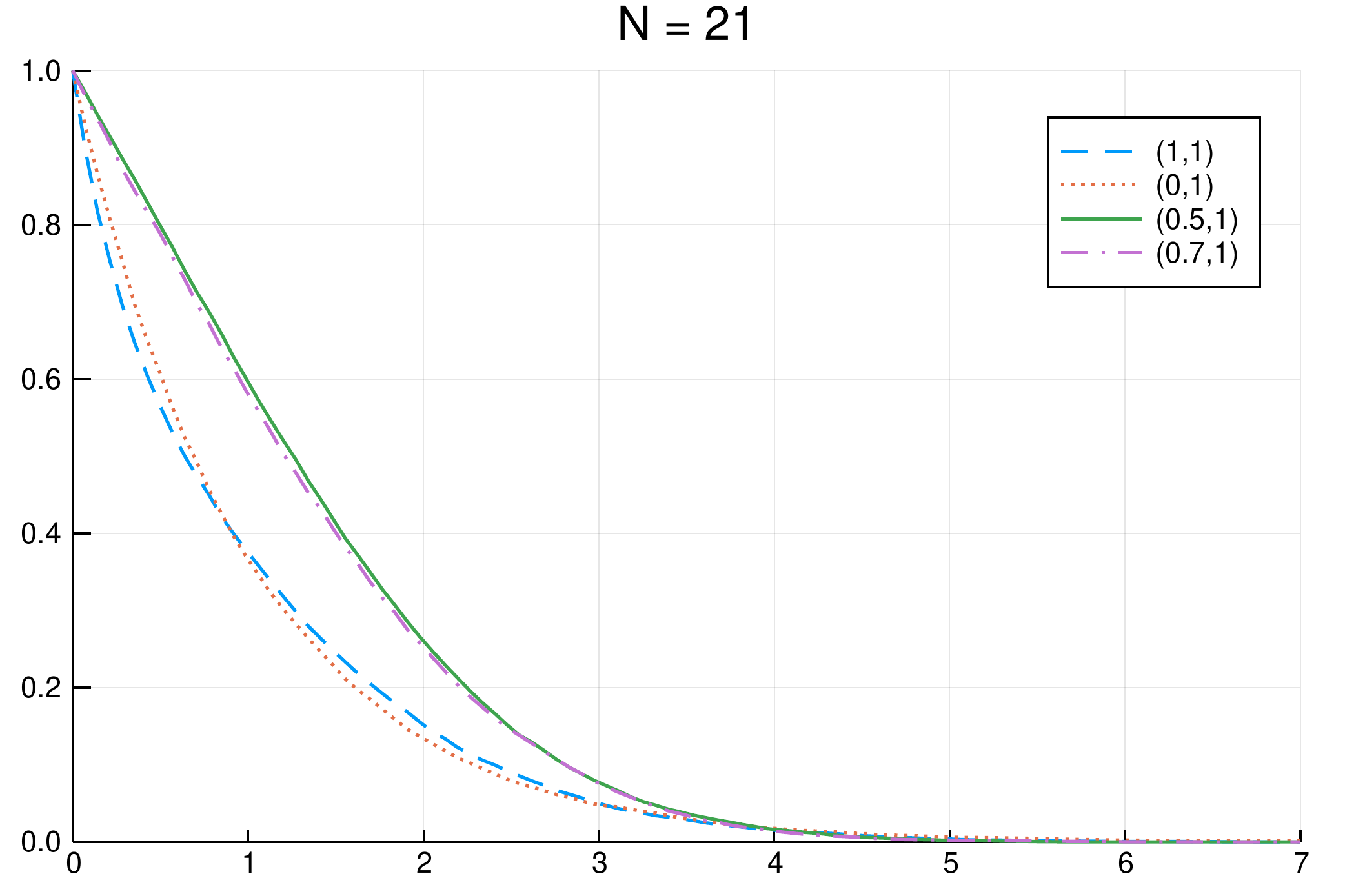}
  \includegraphics[width=.7\textwidth]{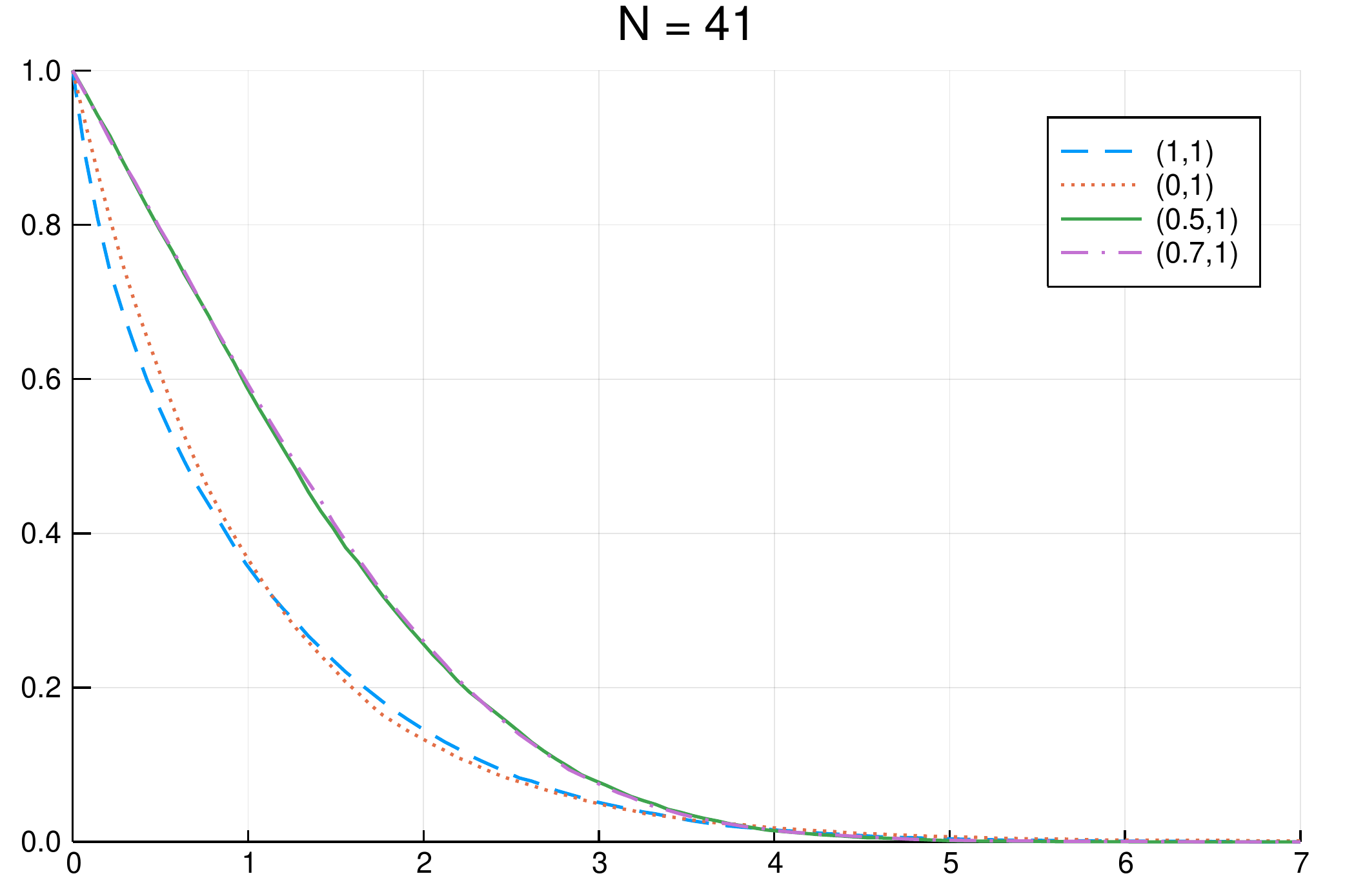}
   \includegraphics[width=.7\textwidth]{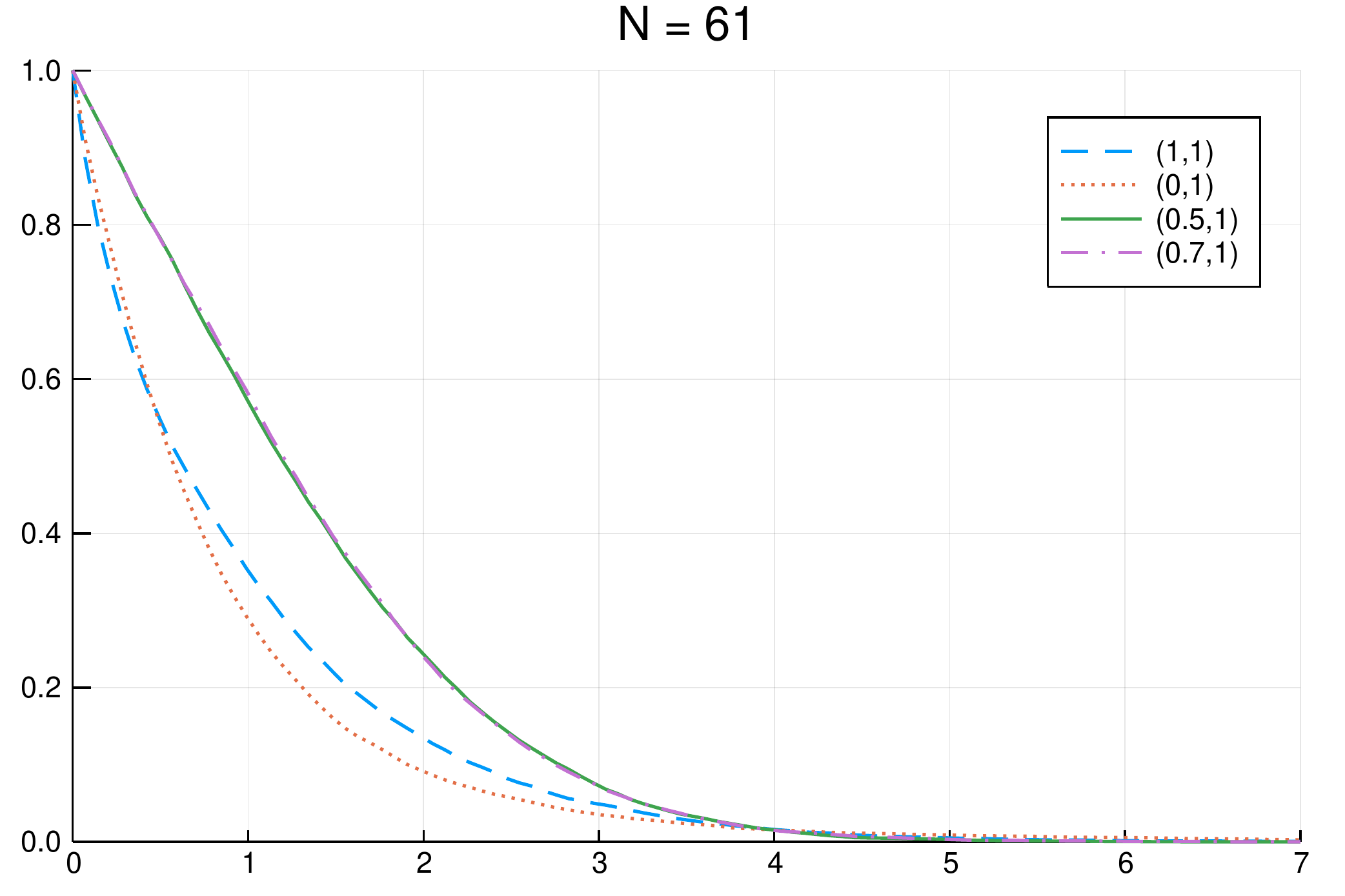}
  \caption{Monte Carlo calculation of the probability that no point satisfying $y = 1$ sampled according to the Coulomb gas on the wedge lies in a neighbourhood of four different points. $N$ is the total number of basis elements and points. We have scaled the statistics so that the variance is one, and have used 10,000 samples. }\label{fig:coulomb}
 \end{center}
\end{figure}

Using the algorithm for sampling determinantal point processes associated with univariate orthogonal polynomials \cite{InvariantEnsembleSampling}, which is trivially adapted to the orthogonal polynomials on the wedge, we can sample from this determinantal point process.  We use this algorithm to calculate statistics of the points. In Figure~\ref{fig:detprocess}, we use the sampling algorithm in a Monte Carlo simulation to approximate the probability that no eigenvalue is present in a neighbourhood of three points for $\alpha = \beta = \gamma = 0$. That is, we take 10,000 samples of a determinantal point process, and calculate the distance of the nearest point to $z_0$, for $z_0$ equal to  $(1,1)$, $(0,1)$, $(0.5,1)$ and $(0.7,1)$.  The plots are  of a complementary empirical cumulative distribution function of these samples.  This gives an estimation of the probability that no eigenvalue is in a neighbourhood of $z_0$.  We have scaled the distributions so that the empirical variance is one: this ensures that the distributions tend to a limit as $N$ becomes large, which is the regime where universality is present. 

\begin{figure}
 \begin{center}
 \includegraphics[width=.7\textwidth]{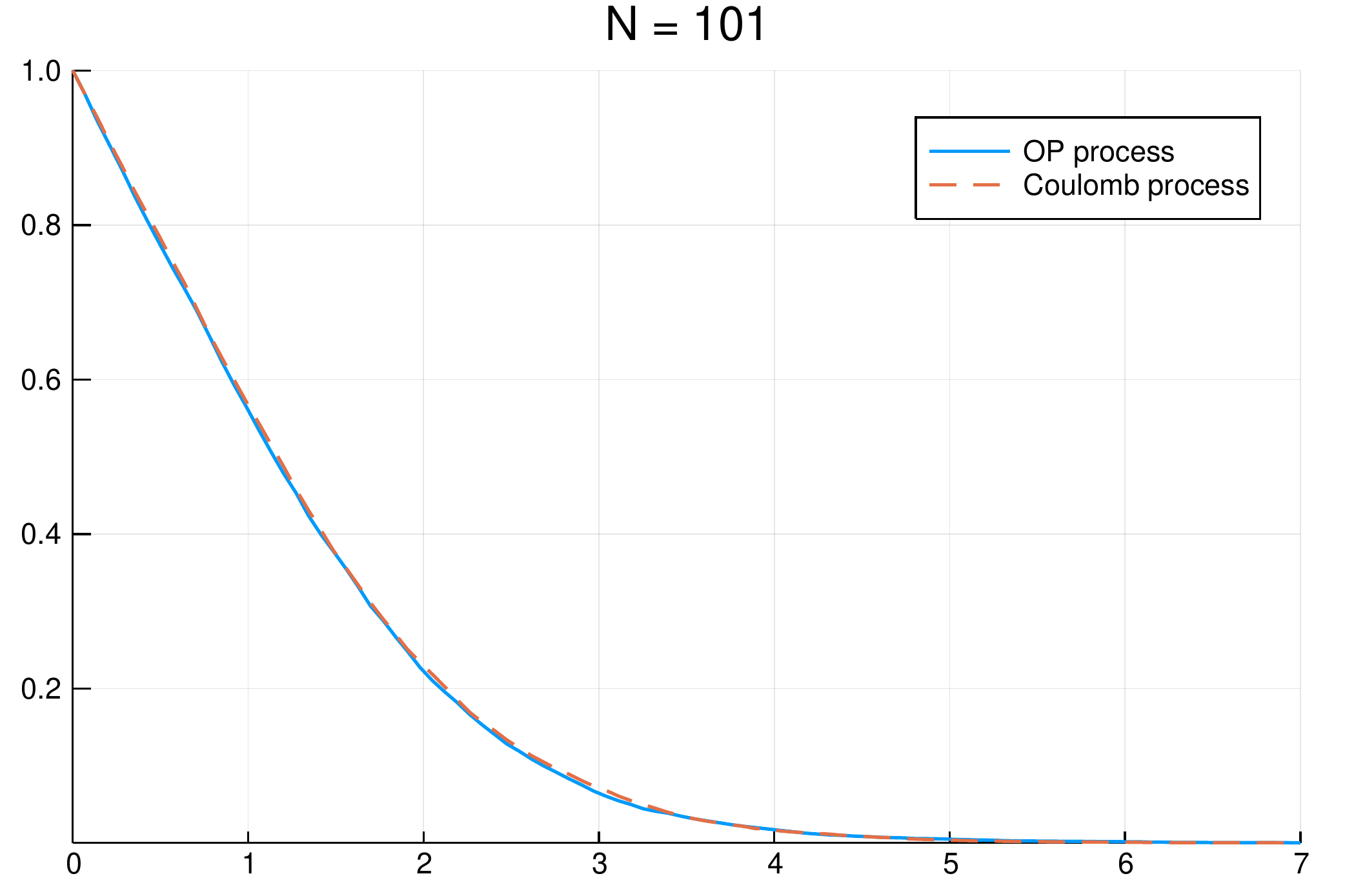}
  \caption{Comparison of the gap probability of the determinantal point process associated to the wedge orthogonal polynomials and Coulomb gas near $(0.5,1)$ for $N = 101$ points. We have scaled the statistics so that the variance is one, and have used 10,000 samples. }\label{fig:opvcoulomb}
 \end{center}
\end{figure}

       In Figure~\ref{fig:coulomb} we plot the same statistics but for samples from the unweighted Coulomb gas on the wedge, which has the distribution 
       	$${1 \over Z_N} \prod_{k < j} \|\lambda_k - \lambda_j\|^2 $$
for 	$\lambda_k$ supported on the wedge.  As this is a Vandermonde determinant squared, it is also a determinantal point process with the basis arising from orthogonalized complex-valued polynomials $1, (x + i y), (x+ i y)^2, \ldots$ \cite{NormalMatrixModel}.   We approximate this orthogonal basis using the modified Gram--Schmidt algorithm with  the wedge inner product calculated via Clenshaw--Curtis quadrature. Again, this fits naturally into the sampling algorithm of \cite{InvariantEnsembleSampling}, hence we can produce samples of this point process.  What we observe is that, while our determinantal point process is not a Coulomb gas, it appears to be  in the same universality class as the Coulomb gas away from the edge and corner, as the statistics follow the same distribution.  This universality class matches that of the Gaussian Unitary Ensemble, as  seen in Figure~\ref{fig:opvcoulomb} where we compare the three for $N = 50$.

\section{Conclusion}

We have introduced multivariate orthogonal polynomials on the wedge and boundary of a square for some natural choices of weights. We have also generated a complete orthogonal  basis with respect to a suitable weight inside the square.  We have looked at determinantal point process statistics and observed a relationship between the resulting statistics and Coulomb gases, suggesting that, away from the corner and edge, they are in the same universality class.

One of the motivations for this work is to solve singular integral equations and evaluate their solutions on contours that have corners, in other words, to generalized the approach of \cite{RMSSOSIEs}. Preliminary work in this direction is included in Appendix~\ref{sec:Stieltjes}, which shows how the recurrence relationship that our polynomials satisfy can be used to evaluate Stieltjes transforms.

\appendix
\section{Jacobi operators}\label{sec:JacobiOperators}

By necessity, multivariate orthogonal polynomials have block-tridiagonal Jacobi operators corresponding to multiplication by $x$ and $y$.  
  We include here the recurrences associated with the inner product $  \la f,g\ra_{\a,\a,\g}$ (that is, $\b = \a$) that encode the Jacobi operators as they have a particularly simple form.   The following lemma gives a linear combination of our orthogonal polynomials that vanish on $x=1$:

\begin{prop}\label{prop:vanish}
For $\beta = \alpha$, we have
$$ (\alpha+\gamma+2) Q_1(x,y)- P_1(x,y) +(1+\alpha) P_0(x,y) = 2 (\alpha+\gamma+2)(1-x)$$
and for $n=1,2,\ldots,$
\begin{align*}
(n + \gamma + \alpha + 2) & Q_{n + 1}(x,y) - (n + 1) P_{n + 1}(x,y) - (n + \gamma) Q_n(x, y) + (n + a + 1)  P_n( x, y) \cr
&=
 2 (1 - x) (2 n + \gamma + \alpha + 2) P_n^{(\gamma + 1, \alpha)}(2 x - 1)
\end{align*}
\end{prop}

\begin{prop}
Assume $(1-x)(1-y) = 0$. Then
\begin{align*}
(1-x) P_0(x,y) & = {1 \over 2} Q_1(x,y) - {1 \over 2 (2 + \gamma+\alpha) } P_1(x,y) + {1+\gamma \over 2 (2+\gamma+\alpha)} P_0(x,y) \cr
(1-x)P_1(x,y) &= {\gamma + \alpha + 2 \over 2  (4 + \gamma + \alpha)} Q_2( x, y) 
 -
 {\gamma + \alpha + 2 \over (3 + \gamma + \alpha) (4 + \gamma + \alpha)} P_2( x, y) \cr
 &- {1+\alpha \over 
4 + \gamma + \alpha} Q_1( x, y) 
   + {4 + 3 \alpha + \gamma (3 + \gamma + \alpha)\over2 (2 + \gamma + \alpha) (4 + \gamma + \alpha)}
   P_1( x, y) \cr
&   -{(1 + \gamma) (1 + \alpha) \over 2 (2 + \gamma + \alpha) (3 + \gamma + \alpha)}
   P_0( x, y) \cr
(1-x) Q_1(x,y) = & -{1 \over 2 (4 + \gamma + \alpha)} Q_2( x, y) 
 + 
 {1\over (3 + \gamma + \alpha) (4 + \gamma + \alpha)} P_2( x, y) \cr
 &+ {(3 + \gamma) \over 
  2 (4 + \gamma + \alpha)} Q_1( x, y) 
   - {(2 + \gamma) \over (2 + \gamma + \alpha) (4 + \gamma + \alpha)}
   P_1( x, y) \cr
&   + {(1 + \gamma) (2 + \gamma) \over 2 (2 + \gamma + \alpha) (3 + \gamma + \alpha)}
   P_0( x, y)
\end{align*}
and, for $n=2,3,\ldots$,
\begin{align*}
(1 - x) P_n(x,y) = & {(1 + \gamma + \alpha + n) (n + \gamma + \alpha + 2) \over 
  2 (1 + \gamma + \alpha + 2 n) (2 + \gamma + \alpha + 2 n)}
   			Q_{n+1}(x,y) \cr
			&- {(1 + \gamma + \alpha + n) (n + 1)\over 
  2 (1 + \gamma + \alpha + 2 n) (2 + \gamma + \alpha + 2 n)}
   P_{n+1}(x,y) \cr
   & - { (\alpha + n) (1 + \gamma + \alpha + n) (1 + \gamma + \alpha + 2 n) \over (1 + \gamma + \alpha + 
     2 n) (2 + \gamma + \alpha + 2 n) (\gamma + \alpha + 2 n)}
   Q_n(x,y) \cr
   & + {(1 + \gamma) (\gamma + \alpha) + 2 (1 + \gamma + \alpha) n + 2 n^2 \over 
  2  (2 + \gamma + \alpha + 2 n) (\gamma + \alpha + 2 n)}
   P_n(x,y) \cr
   &+ {(n + \alpha) (n + \alpha - 1) \over
  2 (1 + \gamma + \alpha + 2 n) (2 n + \gamma + \alpha)}
   Q_{n-1}(x,y) \cr
   &- {(n + \alpha) (n + \gamma) \over
  2 (1 + \gamma + \alpha + 2 n) (2 n + \gamma + \alpha)}  P_{n-1}(x,y) \cr
  (1 - x) Q_n(x,y) = &- {n (2 + \gamma + \alpha + n) \over 2 (2 + \gamma + \alpha + 2 n) (2 n + \gamma + \alpha + 1)}
   			Q_{n+1}(x,y) \cr
			&+ {n (1 + n) \over 2 (2 + \gamma + \alpha + 2 n) (2 n + \gamma + \alpha + 1)}
   P_{n+1}(x,y) \cr
   & + {(1 + \gamma) (2 + \gamma + \alpha) + 2 (1 + \gamma + \alpha) n + 
  2 n^2 \over 2 (\gamma + \alpha + 2 n) (2 + \gamma + \alpha + 2 n)}
   Q_n(x,y) \cr
   & - {n (1 + \gamma + n) \over (\gamma + \alpha + 2 n) (2 + \gamma + \alpha + 2 n)}
   P_n(x,y) \cr
   &- {(1 + \gamma + n) (\alpha + n - 1)\over 2 (\gamma + \alpha + 2 n) (2 n + \gamma + \alpha + 1)}
   Q_{n-1}(x,y) \cr
   & {(1 + \gamma + n) (\gamma + n)\over  2 (\gamma + \alpha + 2 n) (2 n + \gamma + \alpha + 1)}  P_{n-1}(x,y)
  \end{align*}
\end{prop}

\begin{proof}
The first equation  follows from Proposition~\ref{prop:vanish},  since, for $y = 1$, we have (using \cite[(18.9.5)]{DLMF} to increment the first parameter)
	\begin{align*}
	(2n+\gamma + \alpha + 1)(1 - x) P_n(x,y) = &(n + \gamma + \alpha + 1) (1 - x) P_n^{(\gamma + 1, \alpha)}(      2 x - 1) \cr
	&- (n + \alpha) (1-x)P_{n - 1}^{( \gamma+ 1, \alpha)}( 2 x - 1)
	\end{align*}

The second equation  also follows from Proposition~\ref{prop:vanish},  since, for $y = 1$, we have (using \cite[(18.9.6)]{DLMF} to decrement the first parameter)
	\begin{align*}
	(2n+\gamma + \alpha + 1)(1 - x) Q_n(x,y) = & -n (1 - x) P_n^{(\gamma + 1, \alpha)}(      2 x - 1) \cr
	&+ (n + \gamma + 1) (1-x)P_{n - 1}^{( \gamma+ 1, \alpha)}( 2 x - 1).
	\end{align*}
\end{proof}

The recurrences for  multiplication by $1-y$ follow from the symmetries $P_n(x,y) = P_n(y,x)$ and $Q_n(x,y) = -Q_n(y,x)$.

%
%

\section{Stieltjes transform  of orthogonal polynomials}\label{sec:Stieltjes}

Consider the Stieltjes transform
$${\cal S}_\Omega f(z) = \int_\Omega {f(x,y) \over z - (x + i y)} \ds,$$
where $\ds$ is the arc-length differential.  
Just as in one-dimensions, the Stieltjes transform of weighted multivariate orthogonal polynomials satisfies the same recurrence as the orthogonal polynomials themselves

\begin{prop}
Suppose ${\mathbb P}_n$ are a family of orthogonal polynomials with respect to $w(x,y)$. Then,
for $n = 1, 2, \ldots$,
$$ 
	z {\cal S}_\Omega[{\mathbb P}_n w](z) = {\cal S}_\Omega[\zeta {\mathbb P}_n w](z)
	$$	
In particular, if ${\mathbb P}_n$ satisfies the recurrence relationships
\begin{align*}
	x {\mathbb P}_n &= C_{n}^x	{\mathbb P}_{n-1} + A_{n}^x	{\mathbb P}_n + B_{n}^x	{\mathbb P}_{n+1} \\
	y {\mathbb P}_n &= C_{n}^y	{\mathbb P}_{n-1} + A_{n}^y	{\mathbb P}_n + B_{n}^y	{\mathbb P}_{n+1} 
	\end{align*}	
	then for $A_n^z = A_{n}^x +\I A_{n}^y$, $B_n^z = B_{n}^x +\I B_{n}^y$ and $C_n^z = C_{n}^x +\I C_{n}^y$ we have
\begin{align*}
	z {\cal S}_\Omega[{\mathbb P}_n w](z)&= C_n^z 	{\cal S}_\Omega[{\mathbb P}_{n-1} w](z)+A_n^z	{\cal S}_\Omega[{\mathbb P}_n w](z) +B_{n}^z 	{\cal S}_\Omega[{\mathbb P}_{n+1} w](z)
	\end{align*}		
\end{prop}
\begin{proof}
We will identify ${\mathbb R}^2$ and ${\mathbb C}$ and use the notation $\zeta = x + i y$.
Note that
$$	
	z \int_\Omega { f(\zeta) \over z-\zeta} \ds = 
	 \int_\Omega {(z - \zeta)  f(\zeta) \over z - \zeta} \ds +  \int_\Omega {\zeta f(\zeta ) \over z - \zeta} \ds = 
	 \int_\Omega f(\zeta )  \ds+  {\cal S}_\Omega[\zeta f](z)
	 $$
	The first integral is zero if $f$ is orthogonal to $1$. 
\end{proof}

While this holds true for all families of multivariate orthogonal polynomials, in general, satisfying a single recurrence is not sufficient to determine ${\cal S}_\Omega[{\mathbb P}_n w](z)$. However, since our blocks are square, in our case it is:
\begin{cor}
	If $B_n^z = B_{n}^x +\I B_{n}^y$ is invertible, then 
	\begin{align*}
 	{\cal S}_\Omega[{\mathbb P}_{n+1} w](z)&= z (B_{n}^z)^{-1} {\cal S}_\Omega[{\mathbb P}_n w](z) - 	(B_{n}^z)^{-1} C_n^z 	{\cal S}_\Omega[{\mathbb P}_{n-1} w](z)- 	(B_{n}^z)^{-1} A_n^z	{\cal S}_\Omega[{\mathbb P}_n w](z)
	\end{align*}		
\end{cor}

This means that we can calculate the Stieltjes transform in linear time by solving the recurrence equation, using explicit formulas for the $n=0$ and $n=1$ terms. Unfortunately, the results are numerically unstable for both $z$ on and off the contour. Here we sketch an alternative approach built on (F.W.J.) Olver's and Miller's algorithm, see \cite[Section 3.6]{DLMF} for references in the tridiagonal setting and \cite[Section 2.3]{GautschiOPs} for the equivalent application to calculating Stieltjes transforms of univariate orthogonal polynomials.

For $z$ off the contour, we can successfully and stably calculate the Stieltjes transform using a block-wise version of Olver's algorithm, which is equivalent to solving the $2n +1 \times 2n+1$ block-tridiagonal system
\begin{align*}
	\vc q_0 &= 1 \cr
	C_k^z \vc q_{k-1} + (A_k^z - z I)\vc q_k +  B_k^z \vc q_{k+1}& = 0 \qquad \hbox{for}  k = 1,2,\ldots,n-1 \cr
	\vc q_n &= \begin{pmatrix} 0\cr 0\end{pmatrix}
\end{align*}
where $\vc q_0 \in {\mathbb C}^1$ and $\vc q_k \in {\mathbb C}^2$ for $k=1,2,\ldots,n$.   Then
$${\cal S}_\Omega[{\mathbb P}_k w](z)  \approx {\cal S}_\Omega[{\mathbb P}_{0} w](z)  \vc q_k$$
Olver's algorithm consists of performing Gaussian elimination adaptively until a convergence criteria is satisfied.

For $z$ on or near the contour,  we no longer see quick decay in the Stieltjes transform (it is no longer a minimal solution to the recurrence), hence  $n$ must be prohibitively large.  Instead, we adapt Olver's algorithm in a vein similar to Miller's algorithm to allow for a non-decaying tail. We do so by calculating  two additional solutions  $\vc q^1$, and $\vc q^2$ (with the same block-sizes as before) satisfying:
\begin{align*}
	\vc q_0^j &= 1 \cr
	C_k^z \vc q_{k-1}^j + (A_k^z - z I)\vc q_k^j +  B_k^z \vc q_{k+1}^j& = 0 \quad \hbox{for}\quad  k = 1,2,\ldots,n-1 \cr
	 \qquad\vc q_{n}^1 = \begin{pmatrix} 1 \cr 0 \end{pmatrix}  \qquad &\hbox{and}\qquad \vc q_{n}^2 = \begin{pmatrix} 0 \cr 1 \end{pmatrix}.
\end{align*}
These three solutions avoid picking up the exponentially growing solution that forward recurrence does.  Thus we can solve a $3 \times 3$ system for constants $a,b$ and $c$ satisfying
\begin{align*}
a(z) \vc q_0 + b(z) \vc q_0^1 + c(z) \vc q_0^2 = {\cal S}_\Omega[{\mathbb P}_0 w](z) \\
a(z) \vc q_1 + b(z) \vc q_1^1 + c(z) \vc q_1^2 = {\cal S}_\Omega[{\mathbb P}_1 w](z) 
\end{align*}
We immediately have that
\begin{equation}\label{eq:olvermiller}
	{\cal S}_\Omega[{\mathbb P}_k w](z) = a(z) \vc q_k + b(z) \vc q_k^1 + c(z) \vc q_k^2  \qquad\hbox{for}\qquad k  = 0, 1, \ldots, n.
\end{equation}

\begin{figure}\label{fig:olvermiller}
 \begin{center}
 \includegraphics[width=.8\textwidth]{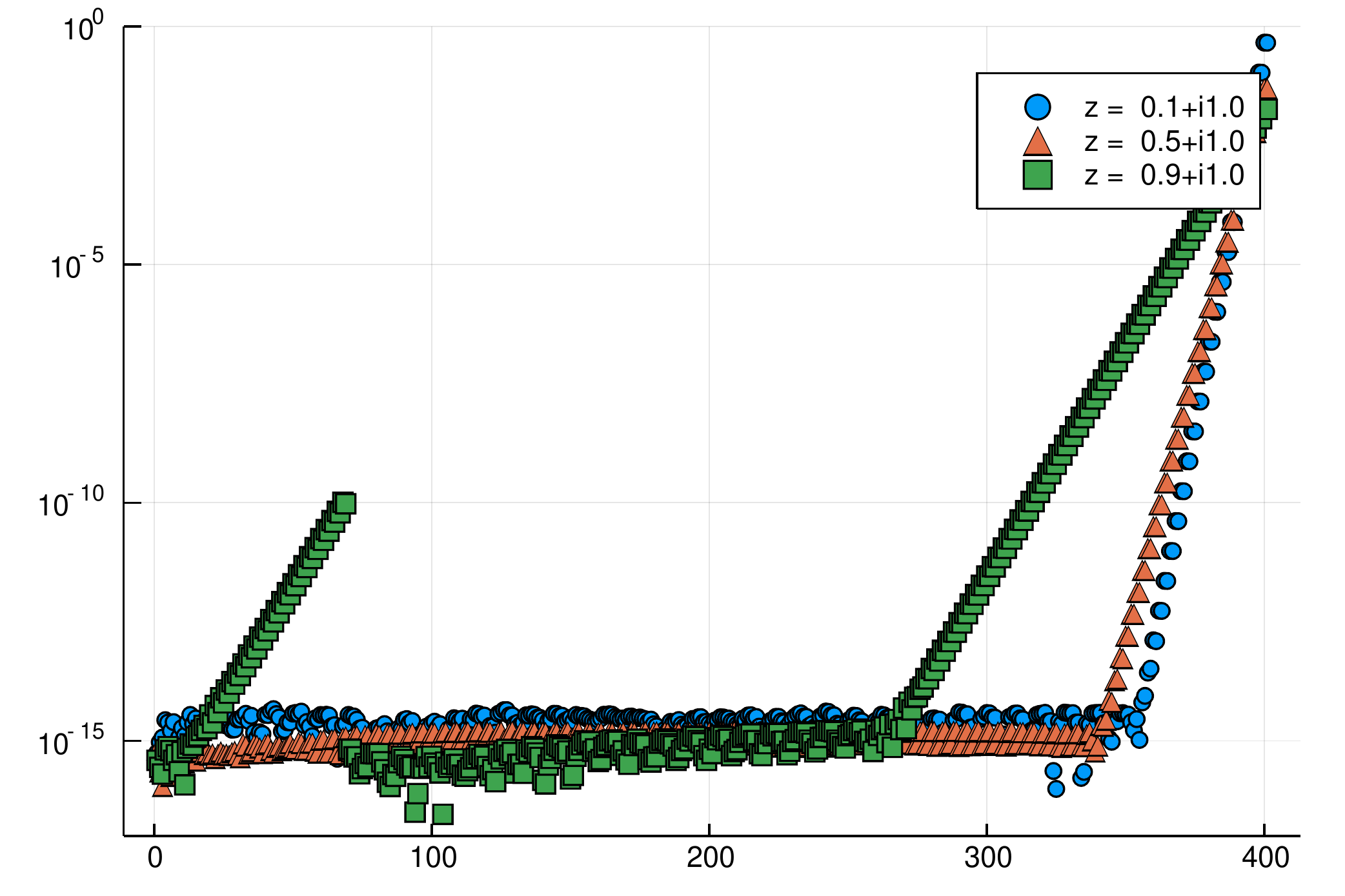}
  \caption{Numerical error in approximating the Stieltjes transform on the wedge, with $ \alpha = \beta = \gamma = 0$ using \eqref{eq:olvermiller}. Note that we need to choose $n$ larger than necessary to avoid the errors in the tail, and there are unresolved numerical errors if $z$ is close to the corner.}
 \end{center}
\end{figure}

While this holds true for all $n$, we note that in practice we need to choose $n$ bigger than the number of coefficients in order to observe numerical stability, see Figure~\ref{fig:olvermiller}. We also find that there are still stability issues near the corner. Resolving these issues is ongoing research.

%

\end{document}